\newif\ifpictures
\picturestrue

\documentclass[11pt]{amsart}
\usepackage{latex_base}
\usepackage{macros}

\author{Mareike Dressler}
\address{Mareike Dressler, School of Mathematics and Statistics, University of New South Wales, Sydney, NSW 2052, Australia.}
\email{m.dressler@unsw.edu.au}

\author{Salma Kuhlmann}
\address{Salma Kuhlmann, Fachbereich Mathematik und Statistik, Universit\"at Konstanz, 78457 Konstanz, Germany.}
\email{salma.kuhlmann@uni-konstanz.de}

\author{Moritz Schick}
\address{Moritz Schick, Fachbereich Mathematik und Statistik, Universit\"at Konstanz, 78457 Konstanz, Germany.}
\email{moritz.schick@uni-konstanz.de}



\title[]{Study of the cone of sums of squares plus sums of nonnegative circuit forms} 

\begin{document}

\begin{abstract}
    In this article, we combine sums of squares (SOS) and sums of nonnegative circuit (SONC) forms,  two independent nonnegativity certificates for real homogeneous polynomials. 
    We consider the convex cone SOS+SONC of forms that decompose into a sum of an SOS and a SONC form and study it from a geometric point of view.
    We show that the SOS+SONC cone is proper and neither closed under multiplications nor under linear transformation of variables. 
    Moreover, we present an alternative proof of an analog of Hilbert's 1888 Theorem for the SOS+SONC cone and prove that in the non-Hilbert cases it provides a proper superset of the union of the SOS and SONC cones. This follows by exploiting a new necessary condition for membership in the SONC cone. 
\end{abstract}

\maketitle
	
\section{Introduction}\label{sec: Introduction}

Studying the convex cone of \emph{nonnegative} or \emph{positive semidefinite (PSD)} polynomials is a central topic in real algebraic geometry and polynomial optimization. 
However, deciding membership in the PSD cone is in general NP-hard and hence intractable for practical purposes, see \cite{MurtyKabadi_NPhard,Bellare_NPHardness}.
Thus, one is interested in \emph{nonnegativity certificates}, i.e. sufficient conditions that imply nonnegativity, which are easier to check. These lead to subcones of the PSD cone for which membership can be decided efficiently via suitable convex programs. 

A classical example of such a nonnegativity certificate is the cone of \emph{sums of squares (SOS)}. Investigating the relationship of the PSD and SOS cones has a long and rich history dating back to David Hilbert.  In a seminal work in 1888 \cite{Hilbert1888}, he completely described all the cases, in terms of number of variables and degree, for which the two cones coincide.  
Consequently, for all other \emph{(non-Hilbert) cases}, there exist PSD polynomials which are not SOS. Due to the nonconstructive nature of his proof, the first such explicit example was only found seven decades later by Motzkin \cite{Motzkin_example} followed by several other examples, see e.g. \cite{Robinson_SomePSDnotSOSPolynomials,BergChristensenJensen_RemarkToMultidimensionalMomentProblem,Schmuedgen_ExampleOfPSDnotSOSpolynomial}. 
Computationally, testing whether a polynomial is SOS can be done via \emph{semidefinite programming (SDP)}. This led to a regained interest in the SOS cone over the last two decades to employ it in polynomial optimization and this SOS/SDP approach has been successfully used in several applications. For an overview, see e.g. \cite{Lasserre_GlobalOptimizationWithPolynomialsAndTheProblemOfMoments,Lasserre_MomentsPosPolysAndTheirApplications,Laurent_survey,BlekhermanEtAl_SDPandConvexAlgGeometry}.
However, for fixed degree there are significantly more nonnegative polynomials than sums of squares as the number of variables tends to infinity \cite{Blekherman_MoreNonnegativePolysThanSOS}, geometrically limiting the SOS cone as inner approximation of the PSD cone. Furthermore, the SOS cone has computational limitations, since SDPs and thus, testing membership in the SOS cone, is expensive for high number of variables or high degrees (see \cite{AhmadiEtAL_ReviewOfSOS}). This motivates exploiting additional structures such as symmetry or sparsity and studying other subcones of the PSD cone. 

A recent nonnegativity certificate specifically suited for sparse polynomials are \emph{sums of nonnegative circuit polynomials (SONC)}, which yield another subcone of the PSD cone. 
This cone was first formally defined in \cite{IlimanDeWolff_Amoebas}, generalizing the earlier concept of \emph{agiforms} in \cite{Reznick_FormsDerivedFromAMGM}. See~\cite{Fidalgo:Kovacec,PanteaEtAL_GlobalInjectivityAndMultipleEquilibria} for independent, related approaches and \cite{ChandrasekaranShah_REPforSignomialOptimization} for an equivalent cone, the cone of \emph{sums of arithmetic geometric exponentials (SAGE)}, in the signomial setting.  
A circuit polynomial has a special structure in terms of its Newton polytope and its coefficients' configuration. Its nonnegativity can be decided by solving a system of linear equations derived from the \emph{inequality of arithmetic and geometric means (AM-GM inequality)}. 
Membership in the SONC cone can be decided efficiently via \emph{relative entropy programming (REP)}, see \cite{ChandrasekaranShah_REPforSignomialOptimization,Dressler:Iliman:deWolff:REP,MurrayChandrasekaranWierman_NewtonPolytopesAndREP}. Indeed, SONC and SAGE approaches have been applied successfully to unconstrained and constrained polynomial and signomial optimization and applications thereof, see e.g. \cite{PanteaEtAL_GlobalInjectivityAndMultipleEquilibria,ChandrasekaranShah_REPforSignomialOptimization,Murray:Chandrasekaran:Wierman,Dressler:Murray}. Other approaches for using the SONC cone in polynomial optimization have been made with \emph{geometric programming} \cite{IlimanDeWolff_LowerBounds}, \emph{second order cone programming}  \cite{MagronWang_SecondOrderConeProgramming}, or \emph{duality theory} \cite{Papp_DualitySONC}.
In \cite[Theorem~8.2]{ForsgardDeWolff_AlgebraicBoundary}, a combinatorial characterization of the equality of the SONC and PSD cones is given for fixed support sets. For our purposes however, we consider the cones in a Hilbert setting, i.e. in terms of number of variables and degree, allowing arbitrary support sets.

In this paper, we combine both certificates and study this combination from a geometrical point of view. 
For all non-Hilbert cases, \cite[Proposition~7.2]{IlimanDeWolff_Amoebas} and \cite[Theorem~3.1]{Dressler_RealZeros} show that the SOS and SONC cones are independent of each other in the sense that none of them contains the other, see Section~\ref{sec: Preliminaries} for more details. Thus, a `better' inner approximation of the PSD cone is given by the Minkowski sum of the SOS and SONC cones or equivalently their convex hull, see \cite[Chapter~6]{Dressler_PhDThesis} for an in-depth motivation. We call this resulting set the \emph{SOS+SONC cone}, which we formally introduce in Section~\ref{section:SOS+SONCIntroductionAndBasicProperties}. Since all cones of interest (PSD, SOS, and SONC cone) are invariant under homogenization, see \cite[Proposition~1.2.4]{Marshall_PositivePolynomialsAndSOS} and \cite[Proposition~3.5]{Dressler_RealZeros}, we restrict our work to the homogeneous setting. 

\medskip
Our contribution consists of a fundamental analysis of the SOS+SONC cone. 
We first prove, that it is proper (Theorem~\ref{thm:SOS+SONCProperAndFulldim'l}), study some of its basic geometric properties, and show that it is neither closed under multiplication nor under linear transformation of variables, see Proposition~\ref{prop:NotMultiplicativelyClosed} and~\ref{prop:NotClosedLinearTransformations}.
In Proposition \ref{prop:ReductionStrategySOS+SONC}, we establish a \emph{reduction strategy} similar as in \cite[Theorem~6.2]{Rajwade_Squares}, which is a key ingredient for some of the proofs in Section~\ref{section:SOS+SONCIntroductionAndBasicProperties} and \ref{sec: MainResults}.
A first study of the relationship of the SOS+SONC cone and the PSD cone can be found in \cite[Corollary 2.17]{Averkov_OptimalSizeofLMIS} and shows that they coincide precisely in the Hilbert cases. This result therefore can be seen as a Hilbert 1888 analog for the SOS+SONC cone. It follows by investigating \emph{semidefinite extensions} of the cones. We present an alternative proof of this statement (Proposition~\ref{prop:SOS+SONCHilbertAnalogon}) by providing explicit forms (homogeneous polynomials) which are PSD but not SOS+SONC for all non-Hilbert cases. 
Moreover, we show that for all non-Hilbert cases, the SOS+SONC cone is a proper superset of the union of the SOS and SONC cones, see Proposition~\ref{prop:SOS+SONCnontrivialExtension}. To prove this result, we provide again for all non-Hilbert cases explicit forms, which are SOS+SONC but neither SOS nor SONC. For this purpose, we also introduce in Theorem~\ref{thm:NecessaryCondition} a novel necessary condition for membership in the SONC cone. Additionally, this condition is used to show that no power greater or equal two of the Motzkin form is SONC, see Proposition \ref{prop:PowersOfMotzkin}. 

\medskip
The paper is structured as follows. Section~\ref{sec: Preliminaries} introduces notation, the PSD, SOS, and the SONC cones, and revisits general properties of these cones. Next, we present in Section~\ref{sec: NecessaryCondition} the new necessary condition to check if a given form is SONC. Section~\ref{section:SOS+SONCIntroductionAndBasicProperties} formally introduces our main object of study, the SOS+SONC cone and contains basic geometrical results and central properties. In Section~\ref{sec: MainResults}, we establish an extended Hilbert 1888 analog for the SOS+SONC cone: Section~\ref{subsec: SeparatingSOS+SONCfromPSD}, discusses the separation of the SOS+SONC cone from the PSD cone and Section~\ref{subsec: SeparatingSOS+SONCfromBothSOSandSONC}, studies the separation of the SOS+SONC cone from the union of the SOS and SONC cones. We conclude with a brief discussion in Section~\ref{sec:Conclusion}.

\section{Preliminaries}\label{sec: Preliminaries}

Let $\struc{\N}\coloneqq\{1,2,3,\ldots, \}$, $\struc{\N_0}\coloneqq\N \cup \{0\}$ be the sets of positive and nonnegative integers, respectively. 
For $\alpb \in \N_0^n$ set $\struc{\abs{\alpb}}\coloneqq\alpha_1+\ldots+\alpha_n$.
For $m \in \N$, we write $\struc{[m]}\coloneqq\{1,\ldots, m\}$ for the discrete set containing all integers from $1$ to $m$ and we set $\struc{[0]}\coloneqq\emptyset$. 

Let $m \in \N$. By $\struc{\eb_i} \in \R^m$ $(i \in [m])$, we denote the $i$-th \struc{standard unit vector}. A set $C \subseteq \R^m$ is called a \struc{\emph{convex cone}} if for all $\bm{x}_1, \bm{x}_2 \in C$ and $\lambda_1, \lambda_2 \geq 0$ we have $\lambda_1 \bm{x}_1+\lambda_2\bm{x}_2 \in C$. A convex cone $C$ is called  \struc{\emph{proper}} if it is closed, \struc{\emph{solid}} (i.e. with nonempty interior), and \struc{\emph{pointed}} (i.e. $C \cap -C\subseteq \{0\}$).
A set $\{\bm{x}_0, \ldots, \bm{x}_k\} \subseteq \R^n, k \in \N$, is said to be \struc{\emph{affinely independent}} if $\{\bm{x}_1-\bm{x}_0, \ldots, \bm{x}_k-\bm{x}_0\}$ is linearly independent. 
The \struc{\emph{affine hull}} of a set $S \subseteq \R^m$ is denoted by $\struc{\aff}(S)\coloneqq\left\{\sum_{i=1}^k \lambda_i \bm{x}_i: k \in \N, \bm{x}_i \in S, \lambda_i \in \R, \sum_{i=1}^k \lambda_i=1\right\}$.
A set $U=\bm{v}_0+V$ with $\bm{v}_0 \in \R^m$, $V \subseteq \R^m$ a linear subspace is an \struc{\emph{affine subspace}} of $\R^m$. Equivalently, there are $A \in \R^{m \times m}, \bm{b} \in \R^m$ such that $U=\{\bm{x} \in \R^m: A\cdot \bm{x} =\bm{b}\}$. The \struc{\emph{dimension}} of $U$ is the dimension of $V$ as a subspace of $\R^m$ or equivalently the corank of $A$.
For a set $S \subseteq \R^m$, we denote its \struc{\emph{relative interior}} by \struc{$\relint(S)$}.
The \struc{\emph{convex hull}} of a set $S \subseteq \R^m$ is $\struc{\conv(S)}\coloneqq\left\{\sum_{i=1}^k \lambda_i \bm{x}_i: k \in \N, \bm{x}_i \in S, \lambda_i \geq 0, \sum_{i=1}^k \lambda_i=1\right\}$. 
A set $\Delta \subseteq \R^m$ is called a \struc{\emph{polytope}} if $\Delta= \conv(S)$ for some finite set $S \subseteq \R^m$. 
For a polytope $\Delta$, its set of \struc{\emph{vertices}} is the inclusion-minimal set $\struc{V(\Delta)} \subseteq \Delta$ such that $\conv(V(\Delta))=\Delta$.
If $V(\Delta)$ is finite and affinely independent, for all $\betab \in \relint(\Delta)$, there exist unique \struc{\emph{barycentric coordinates}} \struc{$\lambda_{\alpb \betab}$}$>0$ ($\alpb \in V(\Delta)$) such that $\betab=\sum_{\alpb \in V(\Delta)} \lambda_{\alpb \betab} \alpb$ and $\sum_{\alpb \in V(\Delta)} \lambda_{\alpb \betab}=1$.
In the case $\abs{V(\Delta)}=m+1$, $\Delta \subseteq \R^m$ is called a \struc{\emph{simplex}}. 
For more details on convex geometry, the reader is referred to \cite{BoydVandenberghe_ConvexOptimization} or \cite{Rockafellar_ConvexAnalysis}. 

Throughout the article, we fix $n \in \N$ and write $\xb=(\varx_1, \ldots, \varx_n)^T$ for the vector containing the $n$ variables $\varx_1, \ldots, \varx_n$. For $k \in \N$, let $\struc{H_{n,k}}=\left\{\sum_{\alpb \in \N_0^n, \abs{\alpb}=k}f_{\alpb} \xb^{\alpb}: f_{\alpb} \in \R \right\}$ be the vector space of $n$-variate homogeneous polynomials or \struc{\emph{forms}} of degree $k$. 
The \struc{\emph{support}} of $f \in H_{n,k}$ is \struc{$\supp(f)$}$\coloneqq \{\alpb\in \N_0^n: f_{\alpb} \neq 0\}$. Further \struc{$\New(f)$}$\coloneqq\conv(\supp(f))$ is the \struc{\emph{Newton polytope}} of $f$, and we set \struc{$V(f)$}$\coloneqq V(\New(f))$.
We partition $\supp(f)$ into
\begin{align*}
	\struc{\Sc(f)}\coloneqq \left\{ {\alpb} \in \supp(f): {\alpb} \in (2 \N_0)^n,  f_{{\alpb}}>0 \right\},
\end{align*}
the set of exponents corresponding to \struc{\emph{monomial squares}} and $\struc{\cI(f)}\coloneqq \supp(f) \backslash \Sc(f)$.     
In addition, we define $\struc{\cD(\betab)}\coloneqq \left\{\Delta \subseteq \R^n: \Delta \textup{ a simplex}, {\betab} \in \relint(\Delta), V(\Delta) \subseteq \Sc(f) \right\}$ and $\struc{N(\betab)}\coloneqq\abs{\cD (\betab)} \in \N_0$ for ${\betab} \in \cI(f)$. Further, we set 
\begin{align*}
	\struc{\cR(f)}\coloneqq \left\{{\alpb} \in \Sc(f): {\alpb} \not \in \bigcup\limits_{{\betab} \in \cI(f)} \bigcup\limits_{\Delta \in \cD (\betab)} V(\Delta) \right\}.
\end{align*}
By abuse of notation, for a set $S \subseteq \R^n$ of exponents we often write $\{f_{\alpb}\xb^{\alpb}: \alpb \in S\}$, the set of monomials of a form $f$ with $S \subseteq \supp(f)$. 
The set of \struc{\emph{real zeros}} of $f \in H_{n,k}$ is denoted by $\struc{\cZ(f)}=\{\bm{x} \in \R^n: f(\bm{x})=0\}$. 

\subsection{The PSD, SOS, and SONC cones}\label{subsec:PSDandSOSandSONC}

In addition to a formal definition of all three cones, we present important properties involving their Newton polytopes, which are used in Section~\ref{sec: MainResults} for the proofs of our main results.
\vspace{0.5em}

\begin{definition}
	An $f \in H_{n,2d}$ is \struc{\emph{positive semidefinite (PSD)}} if $f(\bm{x}) \geq 0$ for all $\bm{x} \in \R^n$. We denote by $\struc{P_{n,2d}}$ the set of PSD forms of degree $2d$.
\end{definition} 

In \cite[Lemma of Section 3]{Reznick_ExtremalPSDformsWithFewTerms}, a necessary condition for a form to be PSD is given, which we recall using our notation above.   

\begin{proposition}\label{prop:NewtonPolyPSD}
	For $f\in P_{n,2d}$, the inclusion $V(f) \subseteq \Sc(f)$ holds. 
\end{proposition}

For $f\in P_{n,2d}$, Proposition \ref{prop:NewtonPolyPSD} shows that $\mathcal{I}(f) \subseteq \supp(f) \backslash V(f)$. Hence, all $\betab \in \mathcal{I}(f)$ lie either in the interior of $\New(f)$ or in the relative interior of a face of $\New(f)$. In both cases, there is $V \subseteq V(f) \subseteq \Sc(f)$ such that $\conv(V)$ is a simplex with $\betab \in \relint(\conv(V))$. For this reason $N(\betab)\geq 1$ and we  write $\struc{\cD({\betab})=\left\{\Delta_1, \ldots, \Delta_{N({\betab})}\right\}}$. 

\begin{definition}
	If $f \in H_{n,2d}$ admits a decomposition $f= \sum_{i=1}^s f_i^2$ where $s \in \N, f_i \in H_{n,d}$, we call $f$ a \struc{\emph{sum of squares (SOS)}}. The set of all such forms is denoted by $\struc{\Sigma_{n,2d}}$. 
\end{definition} 

For PSD and SOS decompositions, the following result is known, which can be interpreted as a sparsity property.

\begin{theorem}\cite[Theorem~1]{Reznick_ExtremalPSDformsWithFewTerms}\label{thm:NewtonPolySumAndSOS}
	For $f, g \in P_{n,2d}$, we have $\New(f) \cup \New(g) \subseteq \New(f+g)$. Further, if $f= \sum_{i=1}^s f_i^2$ is SOS, the inclusion $\New(f_i) \subseteq \frac{1}{2} \New(f)$ holds for all $i \in [s]$.
\end{theorem}

\begin{definition}\label{definition:CircuitPolynomial}
	An $f \in H_{n,2d}$ is called a \struc{\emph{circuit (form)}} if it satisfies the following conditions:
	\begin{enumerate}[(C1),leftmargin=*]
		\item $\mathcal{S}(f) = V(f)$ and $V(f)$ is affinely independent. 
		\item $\mathcal{I}(f)=\{\betab\}$ and $\betab \in \relint(\New(f))\quad$ or $\quad \mathcal{I}(f)=\emptyset$.  
	\end{enumerate}
	The monomials $f_{\alpb} \xb^{\alpb}$ corresponding to $\alpb \in V(f)$ are called the \struc{\emph{outer terms}}. Additionally, if $\mathcal{I}(f)=\{\betab\}$, the monomial $f_{\betab} \xb^{\betab}$ is called the \struc{\emph{inner term}} of $f$.
	In this case, since $\New(f)$ is the convex hull of the affinely independent set $V(f)$, there exist unique barycentric corrdinates $\lambda_{\alpb \betab} >0$ $(\alpb \in V(f))$.
	The \struc{\emph{circuit number}} of $f$ is defined as $\struc{\Theta_f}\coloneqq\prod_{\alpb \in V(f)} \left( \frac{f_{\alpb}}{\lambda_{\alpb \betab}} \right)^{\lambda_{\alpb \betab}}$. 
\end{definition}

If $f$ is a circuit, it can be written as $f(\xb)=\sum_{\alpb \in V(f)} f_{\alpb} \varxvec^{\alpb}+ f_{\betab} \varxvec^{\betab}$.
Note that if $\mathcal{I}(f)=\emptyset$, $f$ is a sum of monomial squares.
We also point out that if $\mathcal{I}(f)=\{\betab\}$, the barycentric coordinates satisfy $\lambda_{\alpb \betab} \in \Q$.
The nonnegativity of a circuit form can be decided easily via its circuit number:

\begin{theorem}\cite[Theorem~1.1]{IlimanDeWolff_Amoebas} \label{thm:NonnegativitySingleCircuit}
	Let $f$ be a circuit form. Then the following are equivalent:
	\begin{enumerate}[(i),leftmargin=*]
		\item $f \in P_{n,2d}$.
		\item $\abs{f_{\betab}} \leq \Theta_f$ where $\cI(f)=\{\betab\}$ and $\betab \in \relint(f)$ \quad or \quad $\cI(f)=\emptyset$.
	\end{enumerate}
\end{theorem}

\begin{definition}
	If an $f \in H_{n,2d}$ admits a decomposition  $f=\sum_{i=1}^s f_i$ where $s \in \N$ and every $f_i$ is a nonnegative circuit, we call $f$ a \struc{\emph{sum of nonnegative circuit forms (SONC)}}. The set of all such forms is denoted by $\struc{C_{n,2d}}$. 
\end{definition}

It is well-known that the sets $P_{n,2d}$, $\Sigma_{n,2d}$, and $C_{n,2d}$ are proper, full-dimensional cones inside the vector space $H_{n,2d}$ (cf.\ \cite{BlekhermanEtAl_SDPandConvexAlgGeometry}, \cite{Reznick_SumsOfEvenPorwersOfRealLinearForms}, and \cite{Dressler_PhDThesis}).

\subsection{Separating the PSD cone from the SOS and SONC cones}

All cases of $(n,2d)$ in which the PSD and SOS cones coincide, are classified in ~\cite{Hilbert1888}. For a proof, see also  \cite[Theorem~1.2.6]{Marshall_PositivePolynomialsAndSOS}.

\begin{theorem}[Hilbert 1888]\label{thm:Hilbert1888}
	The identity $\Sigma_{n,2d} =P_{n,2d}$ holds if and only if  $(n,2d) \in \{(2,2d), (n,2), (3,4)\}$. 
\end{theorem}

For convenience, we refer to the cases $(n,2d) \in \{(2,2d), (n,2), (3,4)\}$ as the \struc{\emph{Hilbert cases}}, and all other cases of $(n,2d)$ as the \struc{\emph{non-Hilbert cases}}.
\medskip

Next, we present famous forms in $P_{n,2d} \setminus \Sigma_{n,2d} $ for $(n,2d) \in \{(4,4), (3,6)\}$. For more details, we refer to \cite{Marshall_PositivePolynomialsAndSOS} or \cite{Reznick_OnHilbertsConstructionOfPositivePolynomials}.

\begin{example}\label{ex:PSDnotSOS}
	\begin{enumerate}[(a),leftmargin=*]
		\item The Motzkin form (1967, see \cite{Motzkin_example})
		
		$\struc{M(\varx, \vary,\varz)}\coloneqq\varx^4 \vary^2 + \varx^2 \vary^4 +\varz^6 - 3 \varx^2 \vary^2 \varz^2 \in P_{3,6}  \backslash \Sigma_{3,6} $.
		
		\item The Robinson forms (1969, see \cite{Robinson_SomePSDnotSOSPolynomials})
		\begin{align*}
			&\struc{R_1(\varx,\vary,\varz)}\coloneqq\varx^6+\vary^6+\varz^6-\left(\varx^4\vary^2+\varx^4\varz^2+\vary^4\varx^2+\vary^4\varz^2 +\varz^4\varx^2+\varz^4\vary^2\right) +3\varx^2\vary^2\varz^2 \in P_{3,6}  \backslash \Sigma_{3,6}, \\
			&\struc{R_2(\varx,\vary,\varz,\varw)}\coloneqq\varx^2(\varx-\varw)^2+\vary^2(\vary-\varw)^2+\varz^2(\varz-\varw)^2 +2\varx\vary\varz(\varx+\vary+\varz-2\varw)  \in P_{4,4}  \backslash \Sigma_{4,4}.
		\end{align*}
		
		\item The Choi-Lam forms (1977, see \cite{ChoiLam_ExtremalPSDForms})
		\begin{align*}
			&\struc{Q_1(\varx, \vary, \varz,\varw)}\coloneqq\varx^2 \vary^2 + \varx^2 \varz^2 + \vary^2 \varz^2 +\varw^4- 4 \varw \varx \vary \varz \in P_{4,4}  \backslash \Sigma_{4,4} , \\
			&\struc{Q_2(\varx, \vary, \varz)}\coloneqq\varx^4\vary^2+\vary^4\varz^2+\varz^4\varx^2- 3\varx^2 \vary^2 \varz^2 \in P_{3,6}  \backslash \Sigma_{3,6}.
		\end{align*}
		
		\item The Schm\"udgen form (1979, see \cite{Schmuedgen_ExampleOfPSDnotSOSpolynomial})
		\begin{align*}
			\struc{S(\varx, \vary, \varz)}\coloneqq & 200 \left[(\varx^3-4\varx \varz^2)^2+ (\vary^3-4\vary \varz^2)^2\right] \\
			&+(\vary^2-\varx^2)\varx(\varx+2\varz) \left[ \varx(\varx-2\varz)+2(\vary^2-4\varz^2)\right]  
			\in P_{3,6}  \backslash \Sigma_{3,6} .
		\end{align*}
		
		\item The modified Motzkin form by Berg, Christensen, and Jensen (1979, see \cite{BergChristensenJensen_RemarkToMultidimensionalMomentProblem})
		\begin{align*}
			\struc{\widehat{M}(\varx, \vary, \varz)}\coloneqq \varx^4\vary^2+\varx^2\vary^4 +\varz^6-\varx^2\vary^2\varz^2 \in P_{3,6}  \backslash \Sigma_{3,6} .
		\end{align*}
	\end{enumerate}
\end{example}

The following theorem provides a set-theoretically comparison of the SOS and SONC cones. It shows that for all non-Hilbert cases, the SOS and SONC cones are inner approximations of the PSD cone, which are independent of each other in the sense that none of them is contained in the other. 

\begin{theorem}\label{thm:SONCseparation}
	The SONC cone $C_{n,2d} $ satisfies the following properties:
	\begin{enumerate}[(i),leftmargin=*]
		\item $C_{n,2d}  \subseteq \Sigma_{n,2d} $ if and only if $n=2$ or $2d=2$ or $(n,2d)=(3,4)$.
		\item $C_{2,2} =\Sigma_{2,2} $ and $\Sigma_{n,2}  \not \subseteq C_{n,2} $ for all $n \geq 3$.
		\item $\Sigma_{n,2d}  \not \subseteq C_{n,2d} $ for all $(n,2d)$ with $2d \geq 4$. 
	\end{enumerate}
\end{theorem}
\begin{proof}
	See \cite[Proposition~7.2]{IlimanDeWolff_Amoebas} and \cite[Theorem~3.1]{Dressler_RealZeros}.
\end{proof}

We illustrate the statements of Theorem \ref{thm:SONCseparation} with some concrete examples.
\begin{example}\label{ex:SONC} 
	\begin{enumerate}[(a),leftmargin=*]
		\item Many famous examples of PSD forms which are not SOS turn out to be SONC. Indeed, the Motzkin form, the two Choi-Lam forms and the modified Motzkin form from Example \ref{ex:PSDnotSOS} are all SONC forms. Even more, they are all already nonnegative circuit forms.
		\item It can be shown that both Robinson forms and the Schm\"udgen form are not SONC, see Examples \ref{ex:NecessaryCondition} and \ref{ex:NecessNotSuff}.
		\item The following forms which are SOS but not SONC can be found in \cite[Section~3.1]{Dressler_PhDThesis}:
		\begin{align*}
			p(\xb)& \coloneqq \sum_{i=1}^{n-1} \left( \varx_i^{d-2} \varx_n^2 + 2 \varx_i^{d-1} \varx_n + \varx_i^d \right)^2 \in P_{n,2d} \backslash C_{n,2d} \quad (n \geq 2), \\
			q(\xb)&\coloneqq(\varx_1 \varx_n+ \varx_2 \varx_n + \varx_1 \varx_2)^2 \varx_n^{2d-4} + \sum_{i=3}^{n-1} \varx_i^{2d}  \in P_{n,2d} \backslash C_{n,2d} \quad (n \geq 3).
		\end{align*}
	\end{enumerate}
\end{example}

\subsection{Sparsity-preserving SONC decompositions}\label{subsec: SONCdecompWithoutTermCancellation}

We review results from \cite{Wang_NonnegativePolysAndCircuitPolys} and \cite{MurrayChandrasekaranWierman_NewtonPolytopesAndREP}, which are used in the proof of Theorem \ref{thm:NecessaryCondition}. Those show that as opposed to SOS, SONC forms admit an important additional sparsity-preserving property: for a SONC form $f$ there exists a decomposition in which the support of every nonnegative circuit is contained in $\supp(f)$ and in which there are no term cancellations between the circuits. 

\medskip

Therefore, let $f \in C_{n,2d}$. Then there is a SONC decomposition of the form 
\begin{align}\label{eq:decomp01}
	f(\xb)=\sum_{{\betab} \in \cI(f)} \sum_{k=1}^{N(\betab)} f_{{\betab} k}(\xb) + \sum_{{\alpb} \in \cR(f)} f_{{\alpb}} \xb^{{\alpb}},
\end{align}
where $\New(f_{\betab k})=\Delta_k \in \mathcal{D}(\betab)$ (i.e. $\betab \in \relint(\Delta_k)$, $V(\Delta_k) \subseteq \mathcal{S}(f) \backslash \mathcal{R}(f)$) and $f_{{\betab} k}(\xb)$ is a nonnegative circuit supported on $V(\Delta_k) \cup \{\betab\}$. 
Moreover, a SONC decomposition as in \eqref{eq:decomp01} can be chosen to be \struc{\emph{cancellation free}}, meaning that there is no term cancellation among the $f_{\betab k}$'s (see \cite[Section~5]{Wang_NonnegativePolysAndCircuitPolys} and \cite[Section~5]{MurrayChandrasekaranWierman_NewtonPolytopesAndREP}).  More precisely, each $f_{\betab k}$ can be written as
\begin{align}\label{eq:decomp02}
	f_{{\betab} k}(\xb)=\left[\sum_{{\alpb} \in \Sc(f) \backslash \cR(f)} \mu_{{\alpb} {\betab}}^{(k)} f_{{\alpb}} \xb^{{\alpb}}\right] + \nu_{{\betab}}^{(k)} f_{{\betab}} \xb^{\betab} \hspace{1em} \left( {\betab} \in \cI(f), k \in [N({\betab})] \right).
\end{align}
where $\mu_{{\alpb} {\betab}}^{(k)}, \nu_{\betab}^{(k)} \geq 0$ for all ${\alpb}, {\betab}, k$ and
\begin{align}
	\sum_{{\betab} \in \cI(f)} \sum_{k=1}^{N({\betab})} \mu_{{\alpb} {\betab}}^{(k)} &= 1  \quad \left({\alpb} \in \Sc(f) \backslash \cR(f)\right), \label{eq:decomp03} \\
	\sum_{k=1}^{N({\betab})} \nu_{\betab}^{(k)} &=1 \quad \left({\betab} \in \cI(f)\right). \label{eq:decomp04}
\end{align}
Recall that for all $\betab \in \mathcal{I}(f)$ and $k \in [N(\betab)]$, we have barycentric coordinates $\lambda_{\alpb \betab}^{(k)}>0$ such that $\betab=\sum_{\alpb \in V(\Delta_k)} \lambda_{\alpb \betab}^{(k)} \alpb$ and $\sum_{\alpb \in V(\Delta_k)} \lambda_{\alpb \betab}^{(k)}=1$.
Note that $V(\Delta_k) \subseteq \mathcal{S}(f) \backslash \cR(f)$.
Hence, we further set $\lambda_{\alpb \betab}^{(k)}=0$ for all $\alpb \in \mathcal{S}(f) \backslash (\cR(f) \cup V(\Delta_k))$ and obtain
\begin{align}\label{eq:decomp05}
	\sum_{{\alpb} \in \Sc(f) \backslash \cR(f)} \lambda_{{\alpb} {\betab}}^{(k)} =1 
	\quad \textup{and} \quad
	\sum_{{\alpb} \in \Sc(f) \backslash \cR(f)} \lambda_{{\alpb} {\betab}}^{(k)} {\alpb} = {\betab}.
\end{align}
In addition, the scalars $\lambda_{{\alpb} {\betab}}^{(k)}$ and $\mu_{{\alpb} {\betab}}^{(k)}$ satisfy $\mu_{{\alpb} {\betab}}^{(k)}=0$ if and only if $\lambda_{{\alpb} {\betab}}^{(k)}=0$.

\section{A necessary condition for the membership in the SONC cone}\label{sec: NecessaryCondition}

In this section we derive a necessary condition for a form to be SONC  (see Theorem~\ref{thm:NecessaryCondition}). The proof is based on the existence of sparsity-preserving SONC decompositions in \eqref{eq:decomp01}-\eqref{eq:decomp05} and the weighted AM-GM inequality~\eqref{eq:weightedAMGM}-\eqref{eq:weightedAMGM02}. 
In Example~\ref{ex:NecessaryCondition}, we show that this condition is not satisfied by the Robinson forms, proving that they are not SONC. We further see that the condition is not sufficient. 
In Section~\ref{sec: MainResults}, we use Theorem~\ref{thm:NecessaryCondition} to separate the SOS+SONC cone from the SONC cone. 

\medskip
For convenience we recall the \struc{\emph{weighted AM-GM inequality}} (see \cite[Section 2.5]{Hardy:Littlewood:Polya}):
Let $m \in \N$ and $x_i, w_i \geq 0$ for $i \in [m]$ such that $\sum_{i=1}^m w_i=1$. The following inequality holds:
\begin{align}\label{eq:weightedAMGM}
	\sum_{i=1}^m w_i x_i \geq \prod_{i=1}^m x_i^{w_i}.
\end{align}
Moreover,
\begin{align}\label{eq:weightedAMGM02}
	\sum_{i=1}^m w_i x_i = \prod_{i=1}^m x_i^{w_i} \Leftrightarrow \textup{ all the } x_k \textup{ with } w_k > 0 \textup{ are equal.}
\end{align}

Next, we state the necessary condition using the notation introduced in Section~\ref{subsec: SONCdecompWithoutTermCancellation}.

\begin{theorem}\label{thm:NecessaryCondition}
	Let $f \in C_{n,2d}$. Then the following inequality holds:
	\begin{align}\label{eq:necessaryCond}
		\sum_{{\betab} \in \cI(f)} \abs{f_{{\betab}}} \leq \sum_{{\alpb} \in \Sc(f) \backslash \cR(f)} f_{{\alpb}}.
	\end{align}
	Moreover, if equality holds in \eqref{eq:necessaryCond}, we have in addition
	\begin{align}\label{eq:equalityCondition}
		\forall {\betab} \in \cI(f) \ \forall k \in [N({\betab})] \ \forall {\alpb} \in \Sc(f) \backslash \cR(f): \quad \mu_{{\alpb} {\betab}}^{(k)} f_{\alpb} = \nu_{\betab}^{(k)} \lambda_{{\alpb} {\betab}}^{(k)} \abs{f_{\betab}},
	\end{align}
	where the scalars $\mu_{{\alpb} {\betab}}^{(k)}, \nu_{\betab}^{(k)}, \lambda_{{\alpb} {\betab}}^{(k)} \geq 0$ are given by a SONC decomposition of $f$ as in \eqref{eq:decomp01}-\eqref{eq:decomp05}.
\end{theorem}
\begin{proof}
	Assume $f \in C_{n,2d}$ with SONC decomposition as in \eqref{eq:decomp01}-\eqref{eq:decomp05}.
	In what follows, let ${\betab} \in \cI(f)$ and $k \in [N(\betab)]$ be arbitrary.
	The circuit number of $f_{{\betab} k}$ is given by
	\begin{align*}
		\Theta_{f_{{\betab} k}} = \prod_{{\alpb} \in \Sc(f) \backslash \cR(f)} \left( \frac{\mu_{{\alpb} {\betab}}^{(k)} f_{\alpb}}{\lambda_{{\alpb} {\betab}}^{(k)}}\right)^{\lambda_{{\alpb} {\betab}}^{(k)}},
	\end{align*}
	where we set $\frac{\mu_{{\alpb} {\betab}}^{(k)} f_{\alpb}}{\lambda_{{\alpb} {\betab}}^{(k)}}=1$ if $\lambda_{{\alpb} {\betab}}^{(k)}=\mu_{{\alpb} {\betab}}^{(k)}=0$. Since $f_{{\betab} k}$ is nonnegative, Theorem \ref{thm:NonnegativitySingleCircuit} yields
	\begin{align*}
		\nu_{\betab}^{(k)} \abs{f_{\betab}} \leq \Theta_{f_{{\betab} k}}.
	\end{align*}
	Summing over $k=1, \ldots, N({\betab})$ shows
	\begin{align}\label{eq:proofInequalities01}
		\begin{split}
			\abs{f_{\betab}} 
			& \overset{\eqref{eq:decomp04}}{=} \sum_{k=1}^{N({\betab})} \nu_{\betab}^{(k)} \abs{f_{\betab}}
			\leq \sum_{k=1}^{N({\betab})} \Theta_{f_{{\betab} k}} 
			= \sum_{k=1}^{N({\betab})}  \prod_{{\alpb} \in \Sc(f) \backslash \cR(f)} \left( \frac{\mu_{{\alpb} {\betab}}^{(k)} f_{\alpb}}{\lambda_{{\alpb} {\betab}}^{(k)}}\right)^{\lambda_{{\alpb} {\betab}}^{(k)}} \\
			& \overset{\eqref{eq:weightedAMGM}}{\leq} \sum_{k=1}^{N({\betab})} \sum_{{\alpb} \in \Sc(f) \backslash \cR(f)} \lambda_{{\alpb} {\betab}}^{(k)} \cdot \left( \frac{\mu_{{\alpb} {\betab}}^{(k)} f_{\alpb}}{\lambda_{{\alpb} {\betab}}^{(k)}}\right)
			=  \sum_{k=1}^{N({\betab})} \sum_{{\alpb} \in \Sc(f) \backslash \cR(f)} \mu_{{\alpb} {\betab}}^{(k)} f_{\alpb}.
		\end{split}
	\end{align}
	Finally, summing over all ${\betab} \in \cI(f)$ leads to
	\begin{align*}
		\sum_{{\betab} \in \cI(f)} \abs{f_{\betab}}
		&\leq \sum_{{\betab} \in \cI(f)}  \sum_{k=1}^{N({\betab})} \sum_{{\alpb} \in \Sc(f) \backslash \cR(f)} \mu_{{\alpb} {\betab}}^{(k)} f_{\alpb} \\
		&=\sum_{{\alpb} \in \Sc(f) \backslash \cR(f)} \left(\sum_{{\betab} \in \cI(f)}  \sum_{k=1}^{N({\betab})}  \mu_{{\alpb} {\betab}}^{(k)} \right) f_{\alpb} 
		\overset{\eqref{eq:decomp03}}{=} \sum_{{\alpb} \in \Sc(f) \backslash \cR(f)} f_{\alpb},
	\end{align*}
	as desired.	
	\vspace{0.5em}
	
	For the second statement assume that $\sum_{{\betab} \in \cI(f)} \abs{f_{\betab}}=\sum_{{\alpb} \in \Sc(f) \backslash \cR(f)} f_{\alpb}$, i.e.~\eqref{eq:necessaryCond} is satisfied with equality. Then in particular all inequalities in \eqref{eq:proofInequalities01} must hold with equality. 
	Thus, for all ${\betab} \in \cI(f)$ and $k \in [N({\betab})]$, we have
	\begin{align*}
		\prod_{{\alpb} \in \Sc(f) \backslash \cR(f)} \left( \frac{\mu_{{\alpb} {\betab}}^{(k)} f_{\alpb}}{\lambda_{{\alpb} {\betab}}^{(k)}}\right)^{\lambda_{{\alpb} {\betab}}^{(k)}}
		=
		\sum_{{\alpb} \in \Sc(f) \backslash \cR(f)} \lambda_{{\alpb} {\betab}}^{(k)} \cdot \left( \frac{\mu_{{\alpb} {\betab}}^{(k)} f_{\alpb}}{\lambda_{{\alpb} {\betab}}^{(k)}}\right)
	\end{align*}
	and \eqref{eq:weightedAMGM02} yields that for ${\alpb} \in \Sc(f) \backslash \cR(f)$ with $\mu_{{\alpb} {\betab}}^{(k)}\neq 0$, the scalars $\frac{\mu_{{\alpb} {\betab}}^{(k)} f_{\alpb}}{\lambda_{{\alpb} {\betab}}^{(k)}}$ are equal.
	For convenience, let $\delta_{\betab}^{(k)}\coloneqq \frac{\mu_{{\alpb} {\betab}}^{(k)} f_{\alpb}}{\lambda_{{\alpb} {\betab}}^{(k)}}$ for ${\alpb} \in \Sc(f) \backslash \cR(f)$ with $\mu_{{\alpb} {\betab}}^{(k)}\neq 0$. 
	Rearranging and using $\mu_{{\alpb} {\betab}}^{(k)}=0$ if and only if $\lambda_{{\alpb} {\betab}}^{(k)}=0$ leads to $\mu_{{\alpb} {\betab}}^{(k)} f_{\alpb}=\delta_{\betab}^{(k)} \lambda_{{\alpb} {\betab}}^{(k)}$ for all $\alpb \in \mathcal{S}(f)\backslash \cR(f)$.
	
	Equality in \eqref{eq:proofInequalities01} also shows that $\nu_{\betab}^{(k)} \abs{f_{\betab}}=\Theta_{f_{{\betab} k}}$ and therefore
	\begin{align*}
		\nu_{\betab}^{(k)} \abs{f_{\betab}} = \prod_{{\alpb} \in \Sc(f) \backslash \cR(f)} \left(\delta_{\betab}^{(k)}\right)^{\lambda_{{\alpb} {\betab}}^{(k)}} \overset{\eqref{eq:decomp05}}{=}\delta_{\betab}^{(k)}.
	\end{align*}
	Hence, we have $\mu_{{\alpb} {\betab}}^{(k)} f_{\alpb} =\delta_{\betab}^{(k)} \lambda_{{\alpb} {\betab}}^{(k)} = \nu_{\betab}^{(k)} \lambda_{{\alpb} {\betab}}^{(k)} \abs{f_{\betab}}$ for all $\alpb \in \mathcal{S}(f) \setminus \cR(f)$ and thus
	\begin{align*}
		\forall {\betab} \in \cI(f) \ \forall k \in [N({\betab})] \ \forall {\alpb} \in \Sc(f) \backslash \cR(f): \mu_{{\alpb} {\betab}}^{(k)} f_{\alpb} = \nu_{\betab}^{(k)} \lambda_{{\alpb} {\betab}}^{(k)} \abs{f_{\betab}},
	\end{align*}
	proving \eqref{eq:equalityCondition}.
\end{proof}

From~\eqref{eq:equalityCondition} we can deduce the following useful condition.
\begin{corollary}\label{cor:equalityCondCor}
	If $f \in C_{n,2d}$ and \eqref{eq:necessaryCond} is satisfied with equality, then for all ${\alpb} \in \Sc(f) \setminus \cR(f)$ and ${\betab} \in \cI(f)$ 
	\begin{align}\label{eq:equalityCondCor}
		f_{\alpb} \geq \min_{k\in[N(\betab)]} \lam_{\alpb \betab}^{(k)}\, |f_{\betab}|.
	\end{align}
\end{corollary}

\begin{proof}
	If \eqref{eq:necessaryCond} is satisfied with equality, then we have \eqref{eq:equalityCondition}. Summing the identity in \eqref{eq:equalityCondition} over all $k=1,\ldots,N(\betab)$ and $\betab \in \cI(f)$ and using \eqref{eq:decomp03} shows that for all $\alpb \in \mathcal{S}(f) \backslash \cR(f)$
	\[
	f_{\alpb}=\sum_{{\betab} \in \cI(f)} \sum_{k=1}^{N({\betab})}  \nu_{\betab}^{(k)} \lambda_{{\alpb} {\betab}}^{(k)} \abs{f_{\betab}}.
	\]
	For all $\betab \in \mathcal{I}(f)$ and $k \in [N(\betab)]$, we have $\nu_{\betab}^{(k)} \lambda_{{\alpb} {\betab}}^{(k)} \abs{f_{\betab}}\geq 0$ and hence using \eqref{eq:decomp04} yields
	\[
	f_{\alpb}\geq \sum_{k=1}^{N({\betab})}  \nu_{\betab}^{(k)} \lambda_{{\alpb} {\betab}}^{(k)} \abs{f_{\betab}}
	\geq \min_{k\in[N(\betab)]} \lam_{\alpb \betab}^{(k)}\, |f_{\betab}|.
	\]
\end{proof}

\begin{example} \label{ex:NecessaryCondition}    
	In what follows we demonstrate how to use Theorem~\ref{thm:NecessaryCondition} and Corollary~\ref{cor:equalityCondCor} to show that certain forms are not SONC. For the forms in (a) and (b), condition \eqref{eq:necessaryCond} is not satisfied. In addition, the form in (c) satisfies \eqref{eq:necessaryCond} with equality, but \eqref{eq:equalityCondCor} does not hold.
	\begin{enumerate}[(a), leftmargin=*]			
		\item Consider the two Robinson forms given in Example~\ref{ex:PSDnotSOS}(b).
		From their Newton polytopes, see Figure~\ref{fig:NewtonPolytopesRobinson}, one can easily read of the following:
		\begin{align*}
			\Sc(R_1) \setminus \cR(R_1)&= \{\varx^6, \vary^6, \varz^6\}, \;
			\cR(R_1)=\{3\varx^2\vary^2\varz^2\}, \\
			\cI(R_1)&=\{-\varx^4\vary^2, -\varx^4\varz^2, -\varx^2\vary^4, -\varx^2 \varz^4, -\vary^4\varz^2, -\vary^2\varz^4\}, \\ 
			\Sc(R_2) \setminus \cR(R_2)&=\{\varx^4, \varx^2\varw^2, \vary^4, \vary^2\varw^2, \varz^4, \varz^2 \varw^2\}, \;
			\cR(R_2)=\emptyset, \text{ and } \\
			\cI(R_2)&=\{-2\varx^3\varw, -2\vary^3\varw, -2\varz^3\varw, 2\varx^2\vary\varz, 2\varx\vary^2\varz, 2\varx\vary\varz^2, -4\varx\vary\varz\varw\}.
		\end{align*}
		Thus, we have
		\begin{align*}
			\sum_{{\alpb} \in \Sc(R_1) \backslash \cR(R_1)} (R_1)_{\alpb}=3 \lneq 6 = \sum_{{\betab} \in \cI(R_1)} \abs{(R_1)_{\betab}}, 
			\\
			\sum_{{\alpb} \in \Sc(R_2) \backslash \cR(R_2)} (R_2)_{\alpb}=6 \lneq 16 = \sum_{{\betab} \in \cI(R_2)} \abs{(R_2)_{\betab}}
		\end{align*}
		and by~\eqref{eq:necessaryCond} neither $R_1$ nor $R_2$ is SONC.
		
		\begin{figure}[t]
			\centering
			\tikzset{every picture/.style={line width=0.75pt}} %set default line width to 0.75pt        
			\begin{tikzpicture}[x=0.75pt,y=0.75pt,yscale=-1,xscale=1]
				%uncomment if require: \path (0,350); %set diagram left start at 0, and has height of 350
				
				%Shape: Regular Polygon [id:dp2380578620455649] 
				\draw   (245.45,250.29) -- (40.97,250.29) -- (143.21,73.21) -- cycle ;
				%Straight Lines [id:da010729240263499085] 
				\draw    (143.21,191.27) -- (40.97,250.29) ;
				%Straight Lines [id:da07091333819100831] 
				\draw    (143.21,73.21) -- (143.21,191.27) ;
				%Straight Lines [id:da9090370652846711] 
				\draw    (143.21,191.27) -- (245.45,250.29) ;
				%Shape: Triangle [id:dp7828808431364291] 
				\draw   (450,60) -- (600,150) -- (300,150) -- cycle ;
				%Shape: Triangle [id:dp7824347480841358] 
				\draw   (450,200) -- (600,290) -- (300,290) -- cycle ;
				%Shape: Triangle [id:dp7585504274966071] 
				\draw  [dash pattern={on 4.5pt off 4.5pt}] (450,130) -- (600,220) -- (300,220) -- cycle ;
				%Straight Lines [id:da6493211110871635] 
				\draw    (300,150) -- (300,290) ;
				%Straight Lines [id:da5650767010411393] 
				\draw    (600,150) -- (600,290) ;
				%Straight Lines [id:da11050041448138814] 
				\draw    (450,200) -- (450,230) ;
				%Straight Lines [id:da6146895450065408] 
				\draw    (300,290) -- (400,260) ;
				%Straight Lines [id:da8757819298182647] 
				\draw    (600,290) -- (500,260) ;
				%Shape: Triangle [id:dp1684443977593144] 
				\draw   (450,230) -- (500,260) -- (400,260) -- cycle ;
				%Straight Lines [id:da22940031012140216] 
				\draw  [dash pattern={on 4.5pt off 4.5pt}]  (300,220) -- (450,180) ;
				%Straight Lines [id:da6968359509840658] 
				\draw  [dash pattern={on 4.5pt off 4.5pt}]  (600,220) -- (450,180) ;
				%Straight Lines [id:da9975142141837421] 
				\draw  [dash pattern={on 4.5pt off 4.5pt}]  (450,130) -- (450,180) ;
				%Straight Lines [id:da338645119505282] 
				\draw  [dash pattern={on 0.84pt off 2.51pt}]  (450,60) -- (450,200) ;
				%Shape: Square [id:dp4393497407047289] 
				\draw  [color={rgb, 255:red, 126; green, 211; blue, 33 }  ,draw opacity=1 ][fill={rgb, 255:red, 126; green, 211; blue, 33 }  ,fill opacity=1 ] (445,55) -- (455,55) -- (455,65) -- (445,65) -- cycle ;
				%Shape: Square [id:dp4385990329997056] 
				\draw  [color={rgb, 255:red, 126; green, 211; blue, 33 }  ,draw opacity=1 ][fill={rgb, 255:red, 126; green, 211; blue, 33 }  ,fill opacity=1 ] (445,195) -- (455,195) -- (455,205) -- (445,205) -- cycle ;
				%Shape: Square [id:dp9417401764445659] 
				\draw  [color={rgb, 255:red, 126; green, 211; blue, 33 }  ,draw opacity=1 ][fill={rgb, 255:red, 126; green, 211; blue, 33 }  ,fill opacity=1 ] (595,285) -- (605,285) -- (605,295) -- (595,295) -- cycle ;
				%Shape: Square [id:dp40285016852091604] 
				\draw  [color={rgb, 255:red, 126; green, 211; blue, 33 }  ,draw opacity=1 ][fill={rgb, 255:red, 126; green, 211; blue, 33 }  ,fill opacity=1 ] (295,285) -- (305,285) -- (305,295) -- (295,295) -- cycle ;
				%Shape: Square [id:dp5312978958561247] 
				\draw  [color={rgb, 255:red, 126; green, 211; blue, 33 }  ,draw opacity=1 ][fill={rgb, 255:red, 126; green, 211; blue, 33 }  ,fill opacity=1 ] (595,145) -- (605,145) -- (605,155) -- (595,155) -- cycle ;
				%Shape: Square [id:dp9551236119744915] 
				\draw  [color={rgb, 255:red, 126; green, 211; blue, 33 }  ,draw opacity=1 ][fill={rgb, 255:red, 126; green, 211; blue, 33 }  ,fill opacity=1 ] (295,145) -- (305,145) -- (305,155) -- (295,155) -- cycle ;
				%Shape: Square [id:dp41452847124002856] 
				\draw  [color={rgb, 255:red, 126; green, 211; blue, 33 }  ,draw opacity=1 ][fill={rgb, 255:red, 126; green, 211; blue, 33 }  ,fill opacity=1 ] (138.21,186.27) -- (148.21,186.27) -- (148.21,196.27) -- (138.21,196.27) -- cycle ;
				%Shape: Square [id:dp6126407897727777] 
				\draw  [color={rgb, 255:red, 126; green, 211; blue, 33 }  ,draw opacity=1 ][fill={rgb, 255:red, 126; green, 211; blue, 33 }  ,fill opacity=1 ] (240.45,245.29) -- (250.45,245.29) -- (250.45,255.29) -- (240.45,255.29) -- cycle ;
				%Shape: Square [id:dp8912761127626916] 
				\draw  [color={rgb, 255:red, 126; green, 211; blue, 33 }  ,draw opacity=1 ][fill={rgb, 255:red, 126; green, 211; blue, 33 }  ,fill opacity=1 ] (35.97,245.29) -- (45.97,245.29) -- (45.97,255.29) -- (35.97,255.29) -- cycle ;
				%Shape: Square [id:dp08469811861149878] 
				\draw  [color={rgb, 255:red, 126; green, 211; blue, 33 }  ,draw opacity=1 ][fill={rgb, 255:red, 126; green, 211; blue, 33 }  ,fill opacity=1 ] (138.21,68.21) -- (148.21,68.21) -- (148.21,78.21) -- (138.21,78.21) -- cycle ;
				%Shape: Circle [id:dp6423659895497322] 
				\draw  [color={rgb, 255:red, 208; green, 2; blue, 27 }  ,draw opacity=1 ][fill={rgb, 255:red, 208; green, 2; blue, 27 }  ,fill opacity=1 ] (104.5,250) .. controls (104.5,246.96) and (106.96,244.5) .. (110,244.5) .. controls (113.04,244.5) and (115.5,246.96) .. (115.5,250) .. controls (115.5,253.04) and (113.04,255.5) .. (110,255.5) .. controls (106.96,255.5) and (104.5,253.04) .. (104.5,250) -- cycle ;
				%Shape: Circle [id:dp8588998801200771] 
				\draw  [color={rgb, 255:red, 208; green, 2; blue, 27 }  ,draw opacity=1 ][fill={rgb, 255:red, 208; green, 2; blue, 27 }  ,fill opacity=1 ] (394.5,260) .. controls (394.5,256.96) and (396.96,254.5) .. (400,254.5) .. controls (403.04,254.5) and (405.5,256.96) .. (405.5,260) .. controls (405.5,263.04) and (403.04,265.5) .. (400,265.5) .. controls (396.96,265.5) and (394.5,263.04) .. (394.5,260) -- cycle ;
				%Shape: Circle [id:dp908599853940985] 
				\draw  [color={rgb, 255:red, 208; green, 2; blue, 27 }  ,draw opacity=1 ][fill={rgb, 255:red, 208; green, 2; blue, 27 }  ,fill opacity=1 ] (444.5,180) .. controls (444.5,176.96) and (446.96,174.5) .. (450,174.5) .. controls (453.04,174.5) and (455.5,176.96) .. (455.5,180) .. controls (455.5,183.04) and (453.04,185.5) .. (450,185.5) .. controls (446.96,185.5) and (444.5,183.04) .. (444.5,180) -- cycle ;
				%Shape: Circle [id:dp32212466186305266] 
				\draw  [color={rgb, 255:red, 208; green, 2; blue, 27 }  ,draw opacity=1 ][fill={rgb, 255:red, 208; green, 2; blue, 27 }  ,fill opacity=1 ] (594.5,220) .. controls (594.5,216.96) and (596.96,214.5) .. (600,214.5) .. controls (603.04,214.5) and (605.5,216.96) .. (605.5,220) .. controls (605.5,223.04) and (603.04,225.5) .. (600,225.5) .. controls (596.96,225.5) and (594.5,223.04) .. (594.5,220) -- cycle ;
				%Shape: Circle [id:dp9890036647013245] 
				\draw  [color={rgb, 255:red, 208; green, 2; blue, 27 }  ,draw opacity=1 ][fill={rgb, 255:red, 208; green, 2; blue, 27 }  ,fill opacity=1 ] (444.5,130) .. controls (444.5,126.96) and (446.96,124.5) .. (450,124.5) .. controls (453.04,124.5) and (455.5,126.96) .. (455.5,130) .. controls (455.5,133.04) and (453.04,135.5) .. (450,135.5) .. controls (446.96,135.5) and (444.5,133.04) .. (444.5,130) -- cycle ;
				%Shape: Circle [id:dp9943546729610877] 
				\draw  [color={rgb, 255:red, 208; green, 2; blue, 27 }  ,draw opacity=1 ][fill={rgb, 255:red, 208; green, 2; blue, 27 }  ,fill opacity=1 ] (294.5,220) .. controls (294.5,216.96) and (296.96,214.5) .. (300,214.5) .. controls (303.04,214.5) and (305.5,216.96) .. (305.5,220) .. controls (305.5,223.04) and (303.04,225.5) .. (300,225.5) .. controls (296.96,225.5) and (294.5,223.04) .. (294.5,220) -- cycle ;
				%Shape: Circle [id:dp09022043809111646] 
				\draw  [color={rgb, 255:red, 208; green, 2; blue, 27 }  ,draw opacity=1 ][fill={rgb, 255:red, 208; green, 2; blue, 27 }  ,fill opacity=1 ] (494.5,260) .. controls (494.5,256.96) and (496.96,254.5) .. (500,254.5) .. controls (503.04,254.5) and (505.5,256.96) .. (505.5,260) .. controls (505.5,263.04) and (503.04,265.5) .. (500,265.5) .. controls (496.96,265.5) and (494.5,263.04) .. (494.5,260) -- cycle ;
				%Shape: Circle [id:dp11476065949929781] 
				\draw  [color={rgb, 255:red, 208; green, 2; blue, 27 }  ,draw opacity=1 ][fill={rgb, 255:red, 208; green, 2; blue, 27 }  ,fill opacity=1 ] (444.5,230) .. controls (444.5,226.96) and (446.96,224.5) .. (450,224.5) .. controls (453.04,224.5) and (455.5,226.96) .. (455.5,230) .. controls (455.5,233.04) and (453.04,235.5) .. (450,235.5) .. controls (446.96,235.5) and (444.5,233.04) .. (444.5,230) -- cycle ;
				%Shape: Circle [id:dp5008505222425228] 
				\draw  [color={rgb, 255:red, 208; green, 2; blue, 27 }  ,draw opacity=1 ][fill={rgb, 255:red, 208; green, 2; blue, 27 }  ,fill opacity=1 ] (171.84,250.37) .. controls (171.84,247.33) and (174.3,244.87) .. (177.34,244.87) .. controls (180.37,244.87) and (182.84,247.33) .. (182.84,250.37) .. controls (182.84,253.41) and (180.37,255.87) .. (177.34,255.87) .. controls (174.3,255.87) and (171.84,253.41) .. (171.84,250.37) -- cycle ;
				%Shape: Circle [id:dp5498269672699949] 
				\draw  [color={rgb, 255:red, 208; green, 2; blue, 27 }  ,draw opacity=1 ][fill={rgb, 255:red, 208; green, 2; blue, 27 }  ,fill opacity=1 ] (205.96,191.27) .. controls (205.96,188.23) and (208.42,185.77) .. (211.46,185.77) .. controls (214.5,185.77) and (216.96,188.23) .. (216.96,191.27) .. controls (216.96,194.3) and (214.5,196.77) .. (211.46,196.77) .. controls (208.42,196.77) and (205.96,194.3) .. (205.96,191.27) -- cycle ;
				%Shape: Circle [id:dp8998716827790676] 
				\draw  [color={rgb, 255:red, 208; green, 2; blue, 27 }  ,draw opacity=1 ][fill={rgb, 255:red, 208; green, 2; blue, 27 }  ,fill opacity=1 ] (171.84,132.16) .. controls (171.84,129.12) and (174.3,126.66) .. (177.34,126.66) .. controls (180.37,126.66) and (182.84,129.12) .. (182.84,132.16) .. controls (182.84,135.2) and (180.37,137.66) .. (177.34,137.66) .. controls (174.3,137.66) and (171.84,135.2) .. (171.84,132.16) -- cycle ;
				%Shape: Circle [id:dp4979787173543453] 
				\draw  [color={rgb, 255:red, 208; green, 2; blue, 27 }  ,draw opacity=1 ][fill={rgb, 255:red, 208; green, 2; blue, 27 }  ,fill opacity=1 ] (103.59,132.16) .. controls (103.59,129.12) and (106.05,126.66) .. (109.09,126.66) .. controls (112.12,126.66) and (114.59,129.12) .. (114.59,132.16) .. controls (114.59,135.2) and (112.12,137.66) .. (109.09,137.66) .. controls (106.05,137.66) and (103.59,135.2) .. (103.59,132.16) -- cycle ;
				%Shape: Circle [id:dp7305569568824102] 
				\draw  [color={rgb, 255:red, 208; green, 2; blue, 27 }  ,draw opacity=1 ][fill={rgb, 255:red, 208; green, 2; blue, 27 }  ,fill opacity=1 ] (69.46,191.27) .. controls (69.46,188.23) and (71.92,185.77) .. (74.96,185.77) .. controls (78,185.77) and (80.46,188.23) .. (80.46,191.27) .. controls (80.46,194.3) and (78,196.77) .. (74.96,196.77) .. controls (71.92,196.77) and (69.46,194.3) .. (69.46,191.27) -- cycle ;
				
				% Text Node
				\draw (151,52) node [anchor=north west][inner sep=0.75pt]   [align=left] {$\displaystyle \varx^{6}$};
				% Text Node
				\draw (252,226) node [anchor=north west][inner sep=0.75pt]   [align=left] {$\displaystyle \varz^{6}$};
				% Text Node
				\draw (22,226) node [anchor=north west][inner sep=0.75pt]   [align=left] {$\displaystyle \vary^{6}$};
				% Text Node
				\draw (71,102) node [anchor=north west][inner sep=0.75pt]   [align=left] {$\displaystyle -\varx^{4}\vary^{2}$};
				% Text Node
				\draw (31,162) node [anchor=north west][inner sep=0.75pt]   [align=left] {$\displaystyle -\varx^{2}\vary^{4}$};
				% Text Node
				\draw (181,102) node [anchor=north west][inner sep=0.75pt]   [align=left] {$\displaystyle -\varx^{4}\varz^{2}$};
				% Text Node
				\draw (211,162) node [anchor=north west][inner sep=0.75pt]   [align=left] {$\displaystyle -\varx^{2}\varz^{4}$};
				% Text Node
				\draw (179.34,256) node [anchor=north west][inner sep=0.75pt]   [align=left] {$\displaystyle-\vary^{2}\varz^{4}$};
				% Text Node
				\draw (67,256) node [anchor=north west][inner sep=0.75pt]   [align=left] {$\displaystyle -\vary^{4}\varz^{2}$};
				% Text Node
				\draw (461,32) node [anchor=north west][inner sep=0.75pt]   [align=left] {$\displaystyle \varx^{2} \varw^{2}$};
				% Text Node
				\draw (580,123) node [anchor=north west][inner sep=0.75pt]   [align=left] {$\displaystyle \varz^{2} \varw^{2}$};
				% Text Node
				\draw (281,123) node [anchor=north west][inner sep=0.75pt]   [align=left] {$\displaystyle \vary^{2} \varw^{2}$};
				% Text Node
				\draw (461,189) node [anchor=north west][inner sep=0.75pt]   [align=left] {$\displaystyle \varx^{4}$};
				% Text Node
				\draw (281,267) node [anchor=north west][inner sep=0.75pt]   [align=left] {$\displaystyle \vary^{4}$};
				% Text Node
				\draw (586,267) node [anchor=north west][inner sep=0.75pt]   [align=left] {$\displaystyle \varz^{4}$};
				% Text Node
				\draw (461,109) node [anchor=north west][inner sep=0.75pt]   [align=left] {$\displaystyle -2\varx^{3} \varw$};
				% Text Node
				\draw (301,221) node [anchor=north west][inner sep=0.75pt]   [align=left] {$\displaystyle -2\vary^{3} \varw$};
				% Text Node
				\draw (555,221) node [anchor=north west][inner sep=0.75pt]   [align=left] {$\displaystyle -2\varz^{3} \varw$};
				% Text Node
				\draw (373,265) node [anchor=north west][inner sep=0.75pt]   [align=left] {$\displaystyle -2\varx\vary^{2} \varz$};
				% Text Node
				\draw (458,224) node [anchor=north west][inner sep=0.75pt]   [align=left] {$\displaystyle -2\varx^{2} \vary\varz$};
				% Text Node
				\draw (465,265) node [anchor=north west][inner sep=0.75pt]   [align=left] {$\displaystyle -2\varx\vary\varz^{2}$};
				% Text Node
				\draw (451,162) node [anchor=north west][inner sep=0.75pt]   [align=left] {$\displaystyle -4\varx\vary\varz\varw$};

			\end{tikzpicture}
			\caption{Newton polytopes of $R_1$ (left) and $R_2$ (right). Green squares and red dots correspond to the sets $\mathcal{S}(R_i)$ and $\mathcal{I}(R_i)$, respectively $(i=1,2)$.}
			\label{fig:NewtonPolytopesRobinson}
		\end{figure}
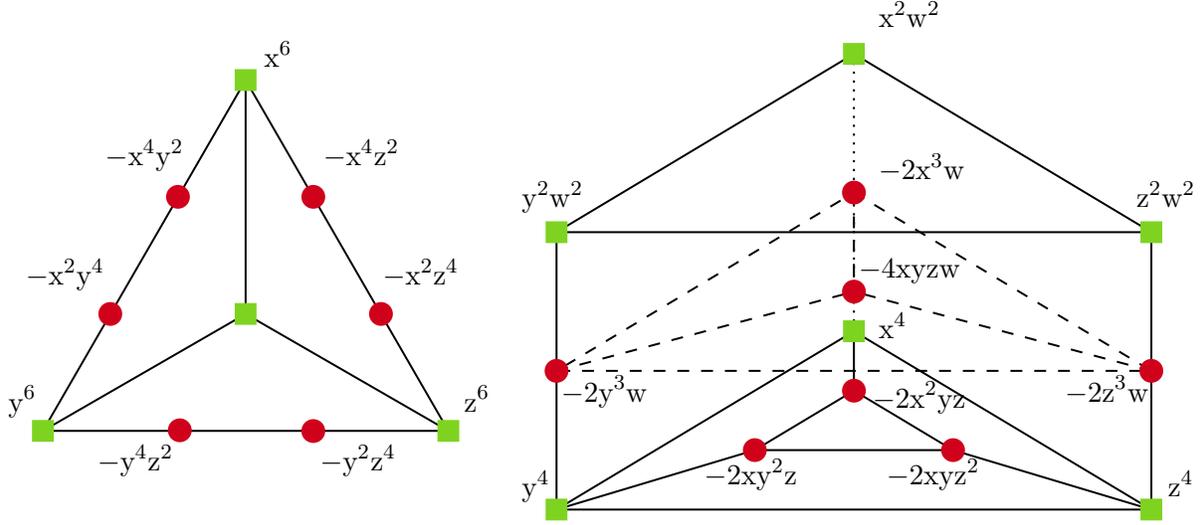
		
		\item Analogously as in (a), one can prove using~\eqref{eq:necessaryCond} that $q(\xb) \not \in  C_{n,2d} $ for $n \geq 3, 2d \geq 4$ (see Example~\ref{ex:SONC}(c)).
		
		\item Let $n \geq 2$.
		Factoring out $p(\xb)$ from Example~\ref{ex:SONC}(c) gives
		\begin{align*}
			p(\xb)& = \sum_{i=1}^{n-1} \varx_i^{2d-4} \varx_n^4 + 6\varx_i^{2d-2} \varx_n^2 + \varx_i^{2d} + 4\varx_i^{2d-3} \varx_n^3+ 4\varx_i^{2d-1} \varx_n,
		\end{align*}
		which yields
		\begin{align*}
			\mathcal{S}(p)\backslash \cR(p)=\bigcup_{i=1}^{n-1} \left\{\varx_i^{2d-4}\varx_n^4, 6\varx_i^{2d-2} \varx_n^2, \varx_i^{2d}\right\}, \quad
			\mathcal{I}(p)= \bigcup_{i=1}^{n-1} \left\{4\varx_i^{2d-3} \varx_n^3, 4\varx_i^{2d-1}\varx_n\right\}
		\end{align*}
		and $\cR(p)=\emptyset$.
		Therefore, \eqref{eq:necessaryCond} is satisfied with equality, since
		\begin{align*}
			\sum_{\betab \in \mathcal{I}(p)} \abs{p_{\betab}}  
			=\sum_{i=1}^{n-1} (4+4) = 8 (n-1)= \sum_{i=1}^{n-1} (1+6+1) = \sum_{\alpb \in \mathcal{S}(p) \backslash \cR(r)} p_{\alpb}.
		\end{align*}
		We now show that \eqref{eq:equalityCondCor} does not hold.
		Assume that $p(\varxvec)$ is SONC and choose a SONC decomposition as in \eqref{eq:decomp01}-\eqref{eq:decomp05}. 
		Consider the monomials $\varx_1^{2d-4}\varx_n^4$ and $4\varx_1^{2d-3} \varx_n^3$ and their corresponding exponents $\alpb=(2d-4)\eb_1+4\eb_n \in \mathcal{S}(p) \setminus \cR(p)$ and $\betab=(2d-3)\eb_1+3\eb_n \in \mathcal{I}(p)$. The only elements of $\cD(\betab)$ are given by $\Delta_1=\conv\{\alpb, (2d-2)\eb_1+2\eb_n\}$ and $\Delta_2=\conv\{\alpb, 2d\eb_1\}$ where we have the convex combinations 
		\begin{align*}
			\betab=\frac{1}{2} \cdot \alpb + \frac{1}{2} \cdot \left((2d-2)\eb_1+2\eb_n\right) =\frac{3}{4} \cdot \alpb + \frac{1}{4} \cdot 2d\eb_1.
		\end{align*}
		Hence
		\begin{align*}
			p_{\alpb}=1 \lneq 2 = \frac{1}{2} \cdot 4 = \min_{k\in[N(\betab)]} \lam_{\alpb \betab}^{(k)}\, |p_{\betab}|,
		\end{align*}
		contradicting \eqref{eq:equalityCondCor} and proving that $p$ is not SONC.
	\end{enumerate}
\end{example}

Recently, Blekherman, Kozhasov and Reznick showed that the odd powers $M^{2m+1}, m \in \N$, of the Motzkin form are not SOS, see \cite{Reznick_OddPowersMotzkin}, \cite{ReznickEtAl_OddPowersMotzkin}. At the \emph{SIAM Conference on Applied Algebraic Geometry (AG23)}, Reznick asked whether or not $M^{2m+1}$ is SONC. In the following, we answer the question in the negative by proving a more general statement.

\begin{proposition}\label{prop:PowersOfMotzkin}
	Let $m \in \N$. The $m$-th power of the Motzkin form $M$ is SONC if and only if $m=1$.
\end{proposition}
\begin{proof}
	The if direction is clear, since $M$ is a nonnegative circuit. For the only if part, let $m \geq 2$. For a contradiction, assume that $M^m \in C_{3,6m}$. 
	Since $M(1,1,1)=0$, we have
	\begin{align*}
		&0=M^m(1,1,1)=\sum_{\alpb \in \supp(M^m)}(M^m)_{\alpb}
		=\sum_{\alpb \in \mathcal{S}(M^m)} \underbrace{(M^m)_{\alpb}}_{\geq 0} + \sum_{\betab \in \mathcal{I}(M^m)} (M^m)_{\betab} \\
		\geq& \sum_{\alpb \in \mathcal{S}(M^m)\backslash \mathcal{R}(M^m)} (M^m)_{\alpb} + \sum_{\betab \in \mathcal{I}(M^m)} (M^m)_{\betab} 
		\geq 
		\sum_{\alpb \in \mathcal{S}(M^m)\backslash \mathcal{R}(M^m)} (M^m)_{\alpb} - \sum_{\betab \in \mathcal{I}(M^m)} \abs{(M^m)_{\betab}} \geq 0,
	\end{align*}
	where the last inequality follows by \eqref{eq:necessaryCond}. Thus, equality must hold everywhere and in particular, $M^m$ satisfies \eqref{eq:necessaryCond} with equality. 
	In what follows, we derive a contradiction to \eqref{eq:equalityCondCor}. 
	Therefore, consider $\alpb\coloneqq m \cdot (0,0,6)^T$ and $\betab\coloneqq (m-1) \cdot (0,0,6)^T+ (2,2,2)^T=(2,2,6m-4)^T$. 
	The only way to obtain the monomial $\varxvec^{\betab}$ from the $m$-th power $M^m$ is via a product $-3\varx^2\vary^2\varz^2 \cdot (\varz^6)^{m-1}$. Counting the number of such products shows $(M^m)_{\betab} =\binom{m}{m-1}\cdot (-3)=-3m<0$. Thus, in particular $\betab \in \mathcal{I}(M^m)$. We claim $\min_{k \in [N(\betab)]}\lambda_{\alpb \betab}^{(k)} \geq \frac{1}{3}$.

	Let $k \in [N(\betab)]$ be arbitrary. First, we show $\alpb \in V(\Delta_k)$. Therefore, let $\tilde{\alpb} \in \mathcal{S}(M^m) \backslash \{\alpb\}$ be arbitrary. Knowing $\supp(M)$, it can be seen that the components of $\tilde{\alpb}$ satisfy $\tilde{\alpha}_1+\tilde{\alpha}_2 \geq 6$. 
	Thus, since $\betab$ is a convex combination of the elements in $V(\Delta_k) \subseteq S(M^m)$ and $\beta_1+\beta_2=4$, we deduce $\alpb \in V(\Delta_k)$. Hence, there is a convex combination
	\begin{align*}
		\betab=\lambda_{\alpb \betab}^{(k)} \alpb + \sum_{i=1}^{\ell} \lambda_{\alpb(i)\betab}^{(k)} \alpb(i),
	\end{align*}
	where $\ell \in \N$ and $V(\Delta_k)=\{\alpb\}\cup \{\alpb(i)\}_{i=1}^\ell\subseteq S(M^m)$ with barycentric coordinates $\lambda_{\alpb \betab}^{(k)}, \lambda_{\alpb(i)\betab}^{(k)} \in (0,1)$.
	Define $\hat{\lambda}\coloneqq \sum_{i=1}^{\ell} \lambda_{\alpb(i)\betab}^{(k)} \in (0,1)$ and  $\hat{\alpb}\coloneqq\sum_{i=1}^{\ell} \frac{\lambda_{\alpb(i)\betab}^{(k)}}{\hat{\lambda}} \alpb(i) \in \conv \{\alpb(i)\}_{i=1}^{\ell} \subseteq \New(M^m)$. As seen above, we have $\alpha(i)_1+\alpha(i)_2 \geq 6$ for all $i \in [\ell]$, yielding $\hat{\alpha}_1+\hat{\alpha}_2 \geq 6$. Thus, the convex combination
	\begin{align*}
		(2,2,6m-4)^T=\betab=\lambda_{\alpb \betab}^{(k)} \alpb+\hat{\lambda}\hat{\alpb}
		= \lambda_{\alpb \betab}^{(k)} \cdot (0,0,6)^T+ \hat{\lambda}\cdot (\hat{\alpha}_1,\hat{\alpha}_2,\hat{\alpha}_3)^T
	\end{align*}
	shows $4=\hat{\lambda} \cdot (\hat{\alpha}_1+\hat{\alpha}_2) \geq 6 \cdot \hat{\lambda}$, i.e. $\hat{\lambda} \leq \frac{2}{3}$. Since $\lambda_{\alpb \betab}^{(k)}+\hat{\lambda}=1$, we conclude $\lambda_{\alpb \betab}^{(k)} \geq \frac{1}{3}$. Since $k \in [N(\betab)]$ was arbitrary, $\min_{k \in [N(\betab)]}\lambda_{\alpb \betab}^{(k)} \geq \frac{1}{3}$.
	Combining above's results and using \eqref{eq:equalityCondCor} shows
	\begin{align*}
		1=(M^m)_{\alpb} \geq \min_{k \in [N(\betab)]}\lambda_{\alpb \betab}^{(k)} \cdot \abs{(M^m)_{\betab}} \geq \frac{1}{3} \abs{(M^m)_{\betab}} = \frac{1}{3} \cdot 3m=m,
	\end{align*}
	which is a contradiction, as $m \geq 2$.
\end{proof}

A generalization of Proposition \ref{prop:PowersOfMotzkin} to other circuits will appear in \cite{Schick_PhDThesis}.

\begin{example}\label{ex:NecessNotSuff}
	We show that \eqref{eq:necessaryCond} is not sufficient.
	\begin{enumerate}[(a),leftmargin=*]
		\item  
		\begin{figure}[t]
			\centering
			\tikzset{every picture/.style={line width=0.75pt}} %set default line width to 0.75pt        
			\begin{tikzpicture}[x=0.75pt,y=0.75pt,yscale=-1,xscale=1]
				%uncomment if require: \path (0,350); %set diagram left start at 0, and has height of 350
				
				%Shape: Regular Polygon [id:dp8830551186673499] 
				\draw   (389.2,201.85) -- (208.27,201.85) -- (298.73,45.16) -- cycle ;
				%Shape: Regular Polygon [id:dp3190812508900993] 
				\draw   (344.75,123.05) -- (298.73,202.75) -- (252.72,123.05) -- cycle ;
				%Shape: Square [id:dp649094553657541] 
				\draw  [color={rgb, 255:red, 126; green, 211; blue, 33 }  ,draw opacity=1 ][fill={rgb, 255:red, 126; green, 211; blue, 33 }  ,fill opacity=1 ] (247.72,118.05) -- (257.72,118.05) -- (257.72,128.05) -- (247.72,128.05) -- cycle ;
				%Shape: Square [id:dp327117058599923] 
				\draw  [color={rgb, 255:red, 126; green, 211; blue, 33 }  ,draw opacity=1 ][fill={rgb, 255:red, 126; green, 211; blue, 33 }  ,fill opacity=1 ] (384.2,196.85) -- (394.2,196.85) -- (394.2,206.85) -- (384.2,206.85) -- cycle ;
				%Shape: Square [id:dp0015147528575329972] 
				\draw  [color={rgb, 255:red, 126; green, 211; blue, 33 }  ,draw opacity=1 ][fill={rgb, 255:red, 126; green, 211; blue, 33 }  ,fill opacity=1 ] (203.27,196.85) -- (213.27,196.85) -- (213.27,206.85) -- (203.27,206.85) -- cycle ;
				%Shape: Square [id:dp8751368664677686] 
				\draw  [color={rgb, 255:red, 126; green, 211; blue, 33 }  ,draw opacity=1 ][fill={rgb, 255:red, 126; green, 211; blue, 33 }  ,fill opacity=1 ] (293.73,40.16) -- (303.73,40.16) -- (303.73,50.16) -- (293.73,50.16) -- cycle ;
				%Shape: Circle [id:dp915982567425526] 
				\draw  [color={rgb, 255:red, 208; green, 2; blue, 27 }  ,draw opacity=1 ][fill={rgb, 255:red, 208; green, 2; blue, 27 }  ,fill opacity=1 ] (339.25,123.05) .. controls (339.25,120.01) and (341.71,117.55) .. (344.75,117.55) .. controls (347.79,117.55) and (350.25,120.01) .. (350.25,123.05) .. controls (350.25,126.09) and (347.79,128.55) .. (344.75,128.55) .. controls (341.71,128.55) and (339.25,126.09) .. (339.25,123.05) -- cycle ;
				%Shape: Circle [id:dp2623359888955725] 
				\draw  [color={rgb, 255:red, 208; green, 2; blue, 27 }  ,draw opacity=1 ][fill={rgb, 255:red, 208; green, 2; blue, 27 }  ,fill opacity=1 ] (293.23,202.75) .. controls (293.23,199.72) and (295.7,197.25) .. (298.73,197.25) .. controls (301.77,197.25) and (304.23,199.72) .. (304.23,202.75) .. controls (304.23,205.79) and (301.77,208.25) .. (298.73,208.25) .. controls (295.7,208.25) and (293.23,205.79) .. (293.23,202.75) -- cycle ;
				
				% Text Node
				\draw (161,199) node [anchor=north west][inner sep=0.75pt]   [align=left] {$\displaystyle 4\varx^{4} \varz^{4}$};
				% Text Node
				\draw (286,19) node [anchor=north west][inner sep=0.75pt]   [align=left] {$\displaystyle 4\varx^{4} \vary^{4}$};
				% Text Node
				\draw (394,199) node [anchor=north west][inner sep=0.75pt]   [align=left] {$\displaystyle \frac{1}{4} \vary^{4} \varz^{4}$};
				% Text Node
				\draw (199,102) node [anchor=north west][inner sep=0.75pt]   [align=left] {$\displaystyle 8\varx^{4} \vary^{2} \varz^{2}$};
				% Text Node
				\draw (341,102) node [anchor=north west][inner sep=0.75pt]   [align=left] {$\displaystyle -2\varx^{2} \vary^{4} \varz^{2}$};
				% Text Node
				\draw (261,212) node [anchor=north west][inner sep=0.75pt]   [align=left] {$\displaystyle -2\varx^{2} \vary^{2} \varz^{4}$};

			\end{tikzpicture}
			
			\caption{Newton polytope of $f$. Green squares and red dots correspond to the exponent sets $\mathcal{S}(f)$ and $\mathcal{I}(f)$, respectively.}
			\label{fig:newtonPolyExNecessNotSuff}
		\end{figure}
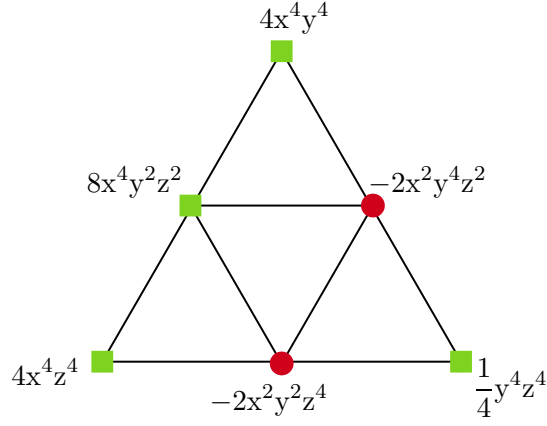
		
		Consider the subsequent SOS form
		\begin{align*}
			f(\varx,\vary,\varz)\coloneqq&\left(2\varx^2\varz^2+2\varx^2\vary^2-\frac{1}{2}\vary^2\varz^2\right)^2
			=& 4 \varx^4 \varz^4 + 4 \varx^4 \vary^4 + \frac{1}{4}  \vary^4 \varz^4+ 8 \varx^4 \vary^2 \varz^2 - 2 \varx^2 \vary^2 \varz^4 - 2 \varx^2 \vary^4 \varz^2.
		\end{align*}
		Figure~\ref{fig:newtonPolyExNecessNotSuff} illustrates $\New(f)$, from which we can read off
		\begin{align*}
			\Sc(f) \setminus \cR(f) &= \left\{ 4\varx^4\varz^4, 4\varx^4\vary^4, \frac{1}{4}\vary^4 \varz^4\right\},\quad	\cR(f)= \left\{8\varx^4\vary^2\varz^2\right\}, \quad 
			\text{and} \\
			\cI(f)&= \left\{ -2\varx^2\vary^2\varz^4, -2\varx^2\vary^4\varz^2\right\}.
		\end{align*}
		Thus we have $\sum_{{\alpb} \in \Sc(f) \backslash \cR(f)} f_{\alpb}= 8.25 \geq 4 = \sum_{{\betab} \in \cI(f)} \abs{f_{\betab}}$, i.e. $f$ satisfies \eqref{eq:necessaryCond}. 
		On the other hand, by the results of Section \ref{subsec: SONCdecompWithoutTermCancellation}, if $f$ is SONC it admits a SONC decomposition $f=8\varx^4\vary^2\varz^2+f_1+f_2$ with
		\begin{align*}
			f_1 \coloneqq 4 \varx^4 \varz^4 + \frac{1}{4} \mu_1 \vary^4 \varz^4 - 2 \varx^2 \vary^2 \varz^4, \quad
			f_2 \coloneqq 4 \varx^4 \vary^4 +  \frac{1}{4} \mu_2 \vary^4 \varz^4 - 2 \varx^2 \vary^4 \varz^2,
		\end{align*}
		and $\mu_1+\mu_2=1, \ \mu_1, \mu_2>0$. Therefore, we have $\Theta_{f_i}=2 \sqrt{\mu_i}$  for $i=1,2$. Since $f_1, f_2$ are nonnegative circuits, by Theorem \ref{thm:NonnegativitySingleCircuit} we must have $2 \sqrt{\mu_i}\geq 2$, which is not possible.
		
		\item Similarly, it can be seen that the Schm\"udgen form $S$ (see Example \ref{ex:PSDnotSOS}(d)) is not SONC but satisfies the necessary condition \eqref{eq:necessaryCond}: First, note $\sum_{{\alpb} \in \Sc(S) \backslash \cR(S)} S_{\alpb}= 6801 \geq 3241 = \sum_{{\betab} \in \cI(S)} \abs{S_{\betab}}$, i.e. $S$ satisfies \eqref{eq:necessaryCond}. Now consider $-1600\vary^4\varz^2 \in \cI(S)$. By the results of Section \ref{subsec: SONCdecompWithoutTermCancellation} and inspecting $\Sc(S)$, if $S$ is SONC there is a SONC decomposition containing a nonnegative circuit $f_1=200a\vary^6+3200b\vary^2\varz^4-1600\vary^4\varz^2$ for some $0<a,b\leq1$. Thus, by Theorem~\ref{thm:NonnegativitySingleCircuit}, we must have $1600 \leq \Theta_{f_1}=1600 \sqrt{ab}$ and hence $a=b=1$. However, as a consequence there cannot be a second circuit $f_2$ in the SONC decomposition of $S$, having $4\varx\vary^4\varz \in \cI(S)$ as an inner term, since $f_2$ would necessarily need an outer term $200c\vary^6, c>0$. Therefore, $S$ cannot be SONC. 
	\end{enumerate}
\end{example}

\section{The cone of sums of squares plus sums of nonnegative circuits}\label{section:SOS+SONCIntroductionAndBasicProperties}

In this section we formally introduce the SOS+SONC cone. We study this cone from a convex geometric point of view and investigate some of its fundamental properties.

\begin{definition}
	A form $f \in H_{n,2d}$ is said to be a \struc{\emph{sum of squares plus a sum of nonnegative circuit forms (SOS+SONC)}} if it decomposes as $f=g+h$ for some $g \in \Sigma_{n,2d}$ and $h \in C_{n,2d}$. Further, we call $\struc{(\Sigma+C)_{n,2d}}\coloneqq\Sigma_{n,2d}+C_{n,2d}$ the \struc{\emph{SOS+SONC cone}} in $n$ variables of degree $2d$. 
\end{definition}

Indeed, $(\Sigma+C)_{n,2d}$ is a convex cone since $(\Sigma+C)_{n,2d} = \conv(\Sigma_{n,2d} \cup C_{n,2d} )$, see \cite[Theorem~3.8]{Rockafellar_ConvexAnalysis}.
We show in Theorem \ref{thm:SOS+SONCProperAndFulldim'l} below that it is in fact proper. 
\medskip

Using the Hahn-Banach Separation Theorem and results from convex analysis, one can show that if $C_1, C_2$ are closed subcones of a pointed cone $K$ in a finite dimensional $\R$-vector space, then the Minkowski sum $C_1+C_2$ is closed (see e.g. \cite{Schick_PhDThesis})\footnote{We thank Greg Blekherman for this observation.}.
However, if one of the cones $C_1, C_2$ is not pointed, then $C_1+C_2$ need not be closed:
\begin{example}
	Consider the  second order cone $C_1\coloneqq\{(x,y,z)\in \R^3: x^2+y^2\leq z^2, z \geq 0\}\subseteq \R^3$ and the half-line $C_2\coloneqq\{\lambda \cdot (1,0,-1): \lambda \geq 0\} \subseteq \R^3$. Then $(0,1,0) \not \in C_1+C_2$ but 
	\begin{align*}
		(0,1,0)=\lim_{\lambda \to \infty} \left( \left(-\lambda,1+\frac{1}{\lambda}, \sqrt{\lambda^2+(1+\frac{1}{\lambda})^2}\right)+(\lambda,0,-\lambda)\right).
	\end{align*}
\end{example}

Below, we prove directly, exploiting the positivity, that the sum of two closed \emph{subsets} of the PSD cone is closed.

\begin{proposition}\label{prop:MinkowskiSumClosedInPSD}
	Let $C_1, C_2 \subseteq P_{n,2d} $ be two closed sets. Then $C_1+C_2 \subseteq P_{n,2d} $ is also closed.
\end{proposition}
\begin{proof}
	Let $(f_k)_{k \in \N} \subseteq C_1+C_2$ be a sequence such that $\lim_{k\to \infty} f_k=f \in H_{n,2d}$ with respect to some norm on $H_{n,2d}$. Note that because $H_{n,2d}$ is a finite dimensional $\R$-vector space, the convergence holds indeed with respect to any norm on $H_{n,2d}$.
	We show that $f \in C_1+C_2$.
	In the following, we consider the supremum norm on the closed hypercube $[0,1]^n$, i.e. 
	\begin{align*}
		\norm{p}_\infty= \sup_{\bm{x} \in [0,1]^n} \abs{p(\bm{x})} \quad \textup{for} \ p \in H_{n,2d}.
	\end{align*}
	For each $k \in \N$, choose $g_k \in C_1$ and $h_k \in C_2$ such that $f_k=g_k+h_k$.
	Since $f_k, g_k$, and $h_k$ are PSD, we have
	\begin{align*}
		\norm{g_k}_\infty
		= \sup_{\bm{x} \in [0,1]^n} \abs{g_k(\bm{x})} 
		= \sup_{\bm{x} \in [0,1]^n} g_k(\bm{x}) 
		\leq \sup_{\bm{x} \in [0,1]^n} f_k(\bm{x})
		= \sup_{\bm{x} \in [0,1]^n} \abs{f_k(\bm{x})}
		= \norm{f_k}_\infty
	\end{align*}
	for all $k \in \N$. Analogously, $\norm{h_k}_\infty \leq \norm{f_k}_\infty$ for all $k \in \N$. Since $(f_k)_{k \in \N} \subseteq H_{n,2d}$ is a convergent sequence, it is in particular bounded. 
	Hence, $(g_k)_{k \in \N}, (h_k)_{k \in \N} \subseteq H_{n,2d}$ are both bounded as well. 
	Because $H_{n,2d}$ is a finite dimensional $\R$-vector space, the boundedness of the sequences implies that both $(g_k)_{k \in \N}$ and $(h_k)_{k \in \N}$ lie in a compact set. 
	Thus, we can choose a convergent subsequence $\left(g_{k(\ell)}\right)_{\ell \in \N}\subseteq C_1$ and some $g \in H_{n,2d}$ such that $\lim_{\ell \to \infty} g_{k(\ell)}=g$. Since $C_1$ is closed, we obtain $g \in C_1$.
	Now, consider the subsequence $\left(h_{k(\ell)}\right)_{\ell \in \N} \subseteq C_2$. Analogously as above there is a convergent subsubsequence $\left(h_{k(\ell)(m)}\right)$ and some $h \in C_2$ such that $\lim_{m \to \infty} h_{k(\ell)(m)}=h$. To sum up, we now have
	\begin{align*}
		\lim_{m \to \infty} f_{k(\ell)(m)}= \lim_{m \to \infty} g_{k(\ell)(m)}+h_{k(\ell)(m)}=g+h \in C_1+C_2.
	\end{align*}
	Finally, since $(f_k)_{k \in \N}$ converges to $f \in H_{n,2d}$, we must have $f=g+h \in C_1+C_2$, as desired. 
\end{proof}

Combining above's observations leads to:
\begin{theorem}\label{thm:SOS+SONCProperAndFulldim'l}
	The convex cone $(\Sigma+C)_{n,2d}$ is proper in $H_{n,2d}$. 
\end{theorem}
\begin{proof}
	The cones $\Sigma_{n,2d}, C_{n,2d} \subseteq H_{n,2d}$ are proper and subsets of $(\Sigma+C)_{n,2d}$. Hence, $(\Sigma+C)_{n,2d}$ is solid and by Proposition~\ref{prop:MinkowskiSumClosedInPSD} closed. Pointedness is clear for all subsets of $P_{n,2d}$.
\end{proof}

In particular, Theorem \ref{thm:SOS+SONCProperAndFulldim'l} shows that $(\Sigma+C)_{n,2d}$ is full-dimensional.
\medskip

The following result is referred to as the \struc{\emph{reduction strategy}} for the SOS+SONC cone. It is motivated by an analogous result for the SOS cone, see \cite[Theorem~6.2]{Rajwade_Squares}, and is an essential ingredient for some of the upcoming proofs.

\begin{proposition}\label{prop:ReductionStrategySOS+SONC}
	Let $f \in H_{n,2d}$ such that $f \not \in (\Sigma+C)_{n,2d}$. For all  $m, \ell \in \N$, we have
	\begin{enumerate}[(i),leftmargin=*]
		\item $f \not \in (\Sigma+C)_{n+m,2d}$,
		\item $\varx_1^{2 \ell} f \not \in (\Sigma+C)_{n,2d+2\ell}$.
	\end{enumerate}
\end{proposition}
\begin{proof}
	Let $f \not \in (\Sigma+C)_{n,2d}$. It suffices to show the claim for $m,\ell=1$ respectively, as the general cases follow inductively.
	
	\begin{enumerate}[(i),leftmargin=*]
		\item For a contradiction assume $f \in (\Sigma+C)_{n+1,2d}$ and choose a decomposition
		\begin{align*}
			f(\xb)=\sum_{i=1}^s g_i^2(\xb, \varx_{n+1})+ \sum_{j=1}^t h_j(\xb, \varx_{n+1}),
		\end{align*}
		where $s,t \in \N_0$, $g_i \in H_{n+1,d}, \ h_j \in H_{n+1,2d} \ (i \in [s], j \in [t])$ and the $h_j$ are nonnegative circuits.
		Without loss of generality, no $h_j(\xb, \varx_{n+1})$ is a sum of monomial squares.
		Now let $j \in [t]$. 
		If $h_j$ does not contain the variable $\varx_{n+1}$,  we have $h_j(\xb,0)= h_j(\xb, \varx_{n+1})$ and $h_j(\xb,0) \in H_{n,2d}$ is a nonnegative circuit. 
		Otherwise, there is an ${\alpb} \in \supp(h_j)$ such that $\alpha_{n+1} \geq 1$. Let $\mathcal{I}(h_j)=\{\betab\}$. We claim that $\beta_{n+1}\geq 1$.
		If $\alpb=\betab$, this is trivial. Else, note that $\betab=\sum_{\alpb \in V(f)} \lambda_{\alpb \betab} \alpb$ with $\lambda_{\alpb \betab}>0$ and $\alpha_{n+1} \geq 1$. This shows $\beta_{n+1} \gneq 0$. Since $\beta_{n+1} \in \N_0$, we obtain $\beta_{n+1}\geq 1$ and thus $h_j(\xb,0)\in H_{n,2d}$ is a sum of monomial squares. In either case, $h_j(\xb,0)\in H_{n,2d}$ is again a nonnegative circuit. Further, we clearly have $g_i(\xb, 0) \in H_{n,d}$.
		Hence, there is a SOS+SONC decomposition
		\begin{align*}
			f(\xb)=\sum_{i=1}^s g_i^2(\xb, 0)+ \sum_{j=1}^t h_j(\xb, 0) \in (\Sigma+C)_{n,2d},
		\end{align*} 
		which is a contradiction. 
		
		\item For a contradiction, assume $\varx_1^2 f \in (\Sigma+ C)_{n,2d+2} $, i.e. there is a decomposition
		\begin{align}\label{proof:PropReductionStrategy}
			\varx_1^2 f= \sum_{i=1}^s g_i^2 + \sum_{j=1}^t h_j
		\end{align}
		for some $s,t \in \N_0$ and forms $g_i \in H_{n,d+1}, h_j \in H_{n,2d+2} \ (i \in [s], j \in [t])$ such that the $h_j$'s are nonnegative circuits. Again, without loss of generality none of the $h_j$ is a sum of monomial squares. 
		Since the left hand side of \eqref{proof:PropReductionStrategy} vanishes at $\varx_1=0$, the right hand side must vanish as well. Therefore, $\varx_1 \mid g_i^2$ and $\varx_1 \mid h_j$ for all $i,j$ as all $g_i^2$ and $h_j$ are PSD. Since $\varx_1$ is prime, we know $\varx_1 \mid g_i$ and $\varx_1^2 \mid g_i^2$. 
		
		We show that $\varx_1^2 \mid h_j$ also holds.
		Since $\varx_1 \mid h_j$, we have $\alpha_1 \geq 1$ for all $\alpb \in \supp(h_j)$.
		As $\alpb \in 2\N_0^n$ for all $\alpb \in V(f)$, we have indeed $\alpha_1 \geq 2$ for all $\alpb \in V(f)$. Let $\mathcal{I}(f)=\{\betab\}$. Since $\betab=\sum_{\alpb \in V(f)} \lambda_{\alpb \betab} \alpb$ and $\sum_{\alpb \in V(f)} \lambda_{\alpb \betab}=1$, we obtain $\beta_1 \geq 2$.
		Thus $\varx_1^2 \mid h_j$.
		Dividing both sides of \eqref{proof:PropReductionStrategy} shows $f \in (\Sigma+C)_{n,2d} $, which is a contradiction. \qedhere
	\end{enumerate}
\end{proof}

Similarly, one can prove a reduction strategy for the SONC cone: If $f \in H_{n,2d}$ and $f \not \in C_{n,2d}$, then  for all $m, \ell \in \N$, we have $f \not \in C_{n+m, 2d}$ and $\varx_1^{2\ell} f \not \in C_{n,2d+2\ell}$. This fact is used in Section \ref{subsec: SeparatingSOS+SONCfromBothSOSandSONC}.

\begin{definition}
	Let $K_{n,2d} \subseteq H_{n,2d}$ be a convex cone for all $d \in \N$. The family $\left(K_{n,2d}\right)_{d \in \N}$ is said to be \struc{\emph{closed under multiplication}} if for all $d \in \N$ and $f, g \in K_{n,2d}$, we have $f \cdot g \in K_{n,4d}$.
\end{definition}

For the sake of simplicity, instead of saying that a family $\left(K_{n,2d}\right)_{d \in \N}$ of cones is closed under multiplication, we just refer to $K_{n,2d}$ being \struc{\emph{closed under multiplication}}.
It is well-known that $\Sigma_{n,2d}$ is closed under multiplication but $C_{n,2d}$ is not, see \cite[Lemma~3.1]{DresslerKurpiszDeWolff_OptimizationOverTheBooleanHypercubeViaSONC}. We now study this behavior for $(\Sigma+C)_{n,2d}$, see Proposition~\ref{prop:NotMultiplicativelyClosed}.
To do so, we are initially interested in the real zeros of nonnegative circuits $f$ on $\R_{>0}^n$, i.e. the sets $\mathcal{Z}(f) \cap \R_{>0}^n$.
For this purpose, we write $\struc{e} \in \R$ for Euler's number. 
A necessary condition for $\cZ(f) \cap \R_{>0}^n \neq \emptyset$ is $\abs{f_{\betab}}=\Theta_f$. After a change of variables, this condition translates without loss of generality to $f_{\betab}=-\Theta_f$.
In \cite[Proposition~3.4]{IlimanDeWolff_Amoebas}, $\cZ(f) \cap \R_{>0}^n$ is characterized for non-homogeneous, nonnegative circuits $f$ where $\bm{0}\coloneqq(0,\ldots,0)^T \in \Sc(f)$ and $\abs{\Sc(f)}=n+1$.
For our purposes, we need a similar result for forms and arbitrary sets $\Sc(f)$. 

\begin{proposition}\label{prop:RealZerosCircuitPolys}
	Let $f \in H_{n,2d}$ be a circuit with $V(f)=\{\alpb(0), \ldots, \alpb(m)\}$, $1 \leq  m \leq n$, $\mathcal{I}(f)=\{\betab\}$ for some $\betab \in \relint(\New(f))$, and $f_{\betab}=-\Theta_f$. Then, for all $\bm{y} \in \R^n$ we have $f(e^{\bm{y}})=0$ if and only if 
	\begin{align}\label{eq:RealZerosCircuitPolys}
		\left( \begin{array}{c}
			(\alpb(1)-\alpb(0))^T \\
			\vdots \\
			(\alpb(m)-\alpb(0))^T 
		\end{array} \right) \bm{y} = \left( \begin{array}{c}
			\log(f_{\alpb(0)}\lambda_{\alpb(1)\betab})- \log(\lambda_{\alpb(0) \betab}f_{\alpb(1)}) \\
			\vdots \\
			\log(f_{\alpb(0)}\lambda_{\alpb(m)\betab})- \log(\lambda_{\alpb(0) \betab}f_{\alpb(m)})
		\end{array}\right).
	\end{align}
	Thus, the set of $\bm{y} \in \R^n$ satisfying \eqref{eq:RealZerosCircuitPolys} is an $n-m$-dimensional affine subspace of $\R^n$.
\end{proposition}
\begin{proof}
	Let $\bm{y} \in \R^n$ be arbitrary and set $\bm{x} \coloneqq e^{\bm{y}}$. We show $f(\bm{x})=0$ if and only if identity \eqref{eq:RealZerosCircuitPolys} holds. Therefore, first assume that $f(\bm{x})=0$, i.e. $0=f(\bm{x})=\sum_{i=0}^m f_{\alpb(i)} \bm{x}^{\alpb(i)}+f_{\betab}\bm{x}^{\betab}=\sum_{i=0}^m f_{\alpb(i)} \bm{x}^{\alpb(i)}-\Theta_f \bm{x}^{\betab}$.
	Rearranging yields 
	\begin{align*}
		\sum_{i=0}^m \lambda_{\alpb(i) \betab} \frac{f_{\alpb(i)} \bm{x}^{\alpb(i)}}{\lambda_{\alpb(i) \betab}} = \sum_{i=0}^m f_{\alpb(i)} \bm{x}^{\alpb(i)}=\Theta_f \bm{x}^{\betab}=\Theta_f \prod_{i=0}^m \bm{x}^{\lambda_{\alpb(i)\betab} \alpb(i)}=\prod_{i=0}^m \left( \frac{f_{\alpb(i)} \bm{x}^{\alpb(i)}}{\lambda_{\alpb(i) \betab}} \right)^{\lambda_{\alpb(i)\betab}}.
	\end{align*}
	By the equality condition \eqref{eq:weightedAMGM02} of the weighted AM-GM inequality, we observe 
	\begin{align*}
		\frac{f_{\alpb(i)} \bm{x}^{\alpb(i)}}{\lambda_{\alpb(i) \betab}} = \frac{f_{\alpb(0)} \bm{x}^{\alpb(0)}}{\lambda_{\alpb(0) \betab}}
	\end{align*}
	or, equivalently,
	\begin{align*}
		e^{\left(\alpb(i)-\alpb(0)\right)^T\bm{y}}= \bm{x}^{\alpb(i)-\alpb(0)} =\frac{f_{\alpb(0)}\lambda_{\alpb(i) \betab}}{f_{\alpb(i)} \lambda_{\alpb(i) \betab}}
	\end{align*}
	for all $i \in [m]$. Thus, taking $\log(\cdot)$ on both sides leads to
	\begin{align*}
		\left(\alpb(i)-\alpb(0)\right)^T\bm{y}=\log \left(f_{\alpb(0)}\lambda_{\alpb(i) \betab}\right)-\log\left(f_{\alpb(i)} \lambda_{\alpb(i) \betab} \right).
	\end{align*}
	Reversing the arguments yields the other direction. The second part of the proposition follows since $\{\alpb(i)\}_{i=0}^m$ is affinely independent and hence, the matrix on the left hand side of \eqref{eq:RealZerosCircuitPolys} has rank $m$.
\end{proof}

\begin{corollary}\label{corollary:RealZerosCircuitPolys}
	Let $f \in P_{n,2d}$ be a nonnegative circuit, $\{\bm{v}_i\}_{i \in [n+1]} \subseteq(\R\backslash\{0\})^n$ such that $f(\bm{v}_i)=0$ for all $i$ and $\{\log \abs{\bm{v}_i}\}_{i \in [n+1]} \subseteq \R^n$ is affinely independent. Then $f=0$.
\end{corollary}
\begin{proof}
	Assume that $f \neq 0$. Note that $f$ cannot be a sum of monomial squares, since those are strictly positive on $(\R\backslash\{0\})^n$. Hence, we have $V(f)=\{\alpb(0), \ldots, \alpb(m)\}$, $1 \leq m \leq n$, and $\mathcal{I}(f)=\{\betab\}$ for some $\betab \in \relint(\New(f))$. 
	As discussed above, since $f$ has zeros, we can further assume without loss of generality that $f_{\betab}=-\Theta_f$, i.e. $f$ is of the form as in Proposition \ref{prop:RealZerosCircuitPolys}.
	Thus, $f$ vanishes at $\bm{w}_i\coloneqq\abs{\bm{v}_i}\coloneqq\left( \abs{v_{i1}},\ldots, \abs{v_{in}} \right)^T \in \R_{>0}^n$ for all $i\in [n+1]$. 
	Clearly, $f(\bm{w})=0$ if and only if $f\left(e^{\log (\bm{w})}\right)=0$ $(\bm{w} \in (\R \backslash \{0\})^n)$. 
	Therefore, all $\log(\bm{w}_i)$ lie in the $n-m$-dimensional affine subspace of $\R^n$ defined by \eqref{eq:RealZerosCircuitPolys}.
	On the other hand, $\log (\bm{w}_i)=\log \abs{\bm{v}_i}$ for all $i \in [n+1]$, thus $\{\log (\bm{w}_i))\}_{i \in [n+1]} \subseteq \R^n$ is affinely independent as well. 
	Hence, we have found an $n+1$-elementary affinely independent subset of an $n-m$-dimensional affine space, contradicting $m \geq 1$.
\end{proof}

As a final ingredient for Proposition \ref{prop:NotMultiplicativelyClosed}, we discuss the following example.

\begin{example}\label{example:RealZerosCircuitPolys}
	Consider $f(\varx,\vary,\varz)\coloneqq\varx+\vary+\varz \in H_{3,1}$ and the Motzkin form $M \in C_{3,6}$.
	Clearly, we have $\cZ(f^2 \cdot M)= \cZ(f) \cup \cZ(M)$. In addition, note
	\begin{align*}
		\cZ(f)&=\{ \lambda \cdot (x,y,1)^T: \lambda \in \R, x+y=-1\} \cup \{(x,y,0)^T: x+y=0\}, \\
		\cZ(M)&=\left\{\lambda \cdot \bm{v}: \lambda \in \R, \bm{v} \in \left\{(\pm 1, \pm 1, \pm 1)^T, (1,0,0)^T, (0,1,0)^T\right\}\right\},
	\end{align*}
	where $(\pm 1, \pm 1, \pm 1)^T$ is any three-tuple having $1$ or $-1$ in its components.
	Assume there is some nonnegative circuit $h \in H_{3,8}$ such that $\cZ(f^2\cdot M) \subseteq \cZ(h)$. 
	Consider the four points
	\begin{align*}
		(1,-2,1)^T,(-2,1,1)^T,(e,e,e)^T,(1,1,1)^T \in (\R \backslash \{0\})^3.
	\end{align*}
	Since either $f$ or $M$ vanishes at each of the four points, $h$ must vanish as well. Moreover, applying pointwise logarithm to the absolute value of each coordinate leads to 
	\begin{align*}
		(0,\log(2),0)^T, (\log(2),0,0)^T, (1,1,1)^T, (0,0,0)^T\in \R^3,
	\end{align*}
	which are affinely independent. By Corollary \ref{corollary:RealZerosCircuitPolys} we have $h=0$.
\end{example}

In the proof below, we use the fact that the SOS cone is \struc{\emph{closed under linear transformation of variables}} in the sense that for all $f \in \Sigma_{n,2d}$ and $A \in \R^{n \times n}$, we have $f(A \cdot \xb)\in \Sigma_{n,2d}$.

\begin{proposition}\label{prop:NotMultiplicativelyClosed}
	The cone $(\Sigma+C)_{n,2d}$ is closed under multiplication if and only if $(n,2d)$ is a Hilbert case. More precisely, for all non-Hilbert cases, there is an $f \in H_{n,d}$ and a nonnegative circuit $g \in H_{n,2d}$ such that $f^2 \cdot g \not \in (\Sigma+C)_{n,4d}$.
\end{proposition}
\begin{proof}
	If $(n,2d)$ is a Hilbert-case, we have $f \cdot g \in \Sigma_{n,4d} \subseteq (\Sigma+C)_{n,4d}$ for all $f,g \in (\Sigma+C)_{n,2d}=\Sigma_{n,2d}$, as $\Sigma_{n,2d}$ is closed under multiplication. Otherwise, let $(n,2d)$ be a non-Hilbert case. We claim that there exist $f,g \in (\Sigma+C)_{n,2d}$ such that $f \cdot g \not \in (\Sigma+C)_{n,4d}$. By the reduction strategy of Proposition~\ref{prop:ReductionStrategySOS+SONC}, it suffices to show the claim for the two base cases $(n,2d)\in \{(3,6),(4,4)\}$: 
	Indeed, if $f,g \in (\Sigma+C)_{n,2d}$ and $f \cdot g \not \in (\Sigma+C)_{n,4d}$, we have $\varx_1^{2\ell}f, \varx_1^{2\ell}g \in (\Sigma+C)_{n+m,2(d+\ell)}$ and $\left(\varx_1^{2\ell} f \right) \cdot \left(\varx_1^{2\ell} g \right)=\varx_1^{4\ell} (fg) \not \in (\Sigma+C)_{n+m,4d+4\ell}$ for all $m, \ell \in \N$. Moreover, it is enough to show that $f_1^2 \cdot M \not \in (\Sigma+C)_{3,8}$ and $f_2^2 \cdot Q_1 \not \in (\Sigma+C)_{4,6}$ for $f_1(\varx,\vary,\varz)\coloneqq\varx+\vary+\varz$ and $f_2(\varx,\vary,\varz, \varw)\coloneqq\varx+\vary+\varz+\varw$, where $M$ and $Q_1$ are the Motzkin and Choi-Lam forms (see Example~\ref{ex:PSDnotSOS}), respectively: 
	Indeed using again the reduction strategy of Proposition~\ref{prop:ReductionStrategySOS+SONC}, the first statement yields that for the base case $(n,2d)=(3,6)$, we have $\varx_1^4 f_1^2 \in \Sigma_{3,6} \subseteq (\Sigma+C)_{3,6}$, $M \in C_{3,6} \subseteq (\Sigma+C)_{3,6}$, and $\left( \varx_1^4 f_1^2\right) M = \varx_1^4 \left( f_1^2 M\right) \not \in (\Sigma+C)_{3,12}$. An analogous observation can be made for the second base case. 
	
	Assume that $f_1^2 \cdot M=f_{\sos}+f_{\sonc} \in (\Sigma+C)_{3,8}$ for some  $f_{\sos} \in \Sigma_{3,8}, f_{\sonc} \in C_{3,8}$.
	Since $f_{\sos}, f_{\sonc}$ are in particular PSD, we have $\cZ(f_{\sos}) \cap \cZ(f_{\sonc}) = \cZ(f_1^2 \cdot M)$, and thus $\cZ(f_1^2 \cdot M) \subseteq \cZ(f_{\sonc})$. 
	Write $f_{\sonc}=\sum_{j=1}^t h_j$ for some $t \in \N$ and nonnegative circuit forms $h_j \in H_{3,8} \ (j \in [t])$. Again, since all $h_j$ are PSD, we have $\bigcap_{j=1}^t \cZ(h_j) =\cZ(f_{\sonc})\supseteq \cZ(f_1^2 \cdot M)$ and hence $\cZ(f_1^2\cdot M) \subseteq \cZ(h_j)$ for all $j \in [t]$. Therefore, Example~\ref{example:RealZerosCircuitPolys} shows $h_j=0$ for all $j \in [t]$ and thus $f_{\sonc}=0$.
	For this reason, $f_1^2 \cdot M=f_{\sos} \in \Sigma_{3,8}$. Now consider the linear transformation of variables defined by the invertible matrix
	\[
	A \coloneqq \left( \begin{array}{rrr}
		1 & -1 & -1 \\
		0 & 1 & 0 \\
		0 & 0 & 1
	\end{array}\right)
	\]
	Since the SOS cone is closed under linear transformations, we have
	\[
	f_{\sos}(A \cdot \xb) = \left( f_1^2 \cdot M \right) (A \cdot \xb) = \left( f_1(A \cdot \xb) \right)^2 \cdot M(A \cdot \xb) = \left( \varx-\vary-\varz+\vary+\varz \right)^2 \cdot M(A \cdot \xb) = \varx^2 \cdot M(A \cdot \xb) \in \Sigma_{3,8}.
	\]
	Using the reduction strategy for the SOS cone, cf.\@ Proposition~\ref{prop:ReductionStrategySOS+SONC} and \cite[Theorem 6.2]{Rajwade_Squares}, we derive that $M(A \cdot \xb) \in \Sigma_{3,6}$. Since the SOS cone is closed under linear transformations of variables, this again implies that $M \in \Sigma_{3,6}$, leading to a contradiction.
	Similarly, it can be seen that $f_2^2\cdot Q_1 \not\in (\Sigma+C)_{4,6}$.
\end{proof}

Proposition \ref{prop:NotMultiplicativelyClosed} shows in particular that the convex cone of $n$-variate polynomials that decompose into a sum of an SOS and a SONC polynomial is neither a preorder nor a quadratic module. 
\medskip

In the proof of Proposition~\ref{prop:NotMultiplicativelyClosed} we utilized the fact that the SOS cone is closed under linear transformations of variables. This is yet another property that holds for the SOS cone but generally does not apply to the SONC cone (see \cite[Corollary~3.2]{DresslerKurpiszDeWolff_OptimizationOverTheBooleanHypercubeViaSONC}).
In Proposition~\ref{prop:NotClosedLinearTransformations}, we study this property for the SOS+SONC cone. For the proof, we use the following lemma. It can be derived from the theory of maximal mediated sets, see \cite{Reznick_FormsDerivedFromAMGM} and \cite{HartzerRoehrigDeWolffYueruek_InitialStepsCLassificationMediatedSets}, or by a direct argument, see below.

\begin{lemma}\label{lemma:PSDCircuitThreeNomial}
	Let $f= f_{\alpb(1)} \xb^{\alpb(1)}+ f_{\alpb(2)} \xb^{\alpb(2)}+f_{\betab} \xb^{\betab} \in C_{n,2d}$ be a proper PSD circuit with inner term $f_{\betab} \xb^{\betab}$ satisfying $\betab=\frac{1}{2}\alpb(1)+\frac{1}{2}\alpb(2)$. Then $f$ is a sum of binomial squares and $f \in \Sigma_{n,2d}$, in particular.
\end{lemma}
\begin{proof}
	By assumption, we have the barycentric coordinates $\lambda_{\alpb(1) \betab}=\lambda_{\alpb(1) \betab}=\frac{1}{2}$. Since $f$ is a PSD circuit, we derive from Theorem~\ref{thm:NonnegativitySingleCircuit}
	\begin{align*}
		\abs{f_{\betab}} \leq \Theta_f= \left(\frac{f_{\alpb(1)}}{\lambda_{\alpb(1) \betab}} \right)^{\lambda_{\alpb(1) \betab}} \cdot \left(\frac{f_{\alpb(2)}}{\lambda_{\alpb(2) \betab}} \right)^{\lambda_{\alpb(2) \betab}}= 2 \sqrt{f_{\alpb(1)} f_{\alpb(2)}}.
	\end{align*}
	After rearranging, we see that $\delta \coloneqq \frac{\abs{f_{\betab}}}{2 \sqrt{f_{\alpb(1)} f_{\alpb(2)}}} \leq 1$.
	In the following, we assume that $f_{\betab} <0$. Otherwise, the statement can be shown similarly. Using $\betab=\frac{1}{2}\alpb(1)+\frac{1}{2}\alpb(2)$ we have 
	\begin{align*}
		f&= f_{\alpb(1)} \xb^{\alpb(1)}+ f_{\alpb(2)} \xb^{\alpb(2)}+f_{\betab} \xb^{\betab}\\
		&= \frac{\abs{f_{\betab}}}{2 \sqrt{f_{\alpb(1)} f_{\alpb(2)}}}\left( \sqrt{f_{\alpb(1)}} \xb^{\frac{1}{2}\alpb(1)}- \sqrt{f_{\alpb(2)}} \xb^{\frac{1}{2}\alpb(2)}\right)^2+ f_{\alpb(1)} \left( 1- \frac{\abs{f_{\betab}}}{2 \sqrt{f_{\alpb(1)} f_{\alpb(2)}}} \right) \xb^{\alpb(1)} \\
		& \phantom{=} + f_{\alpb(2)} \left( 1- \frac{\abs{f_{\betab}}}{2 \sqrt{f_{\alpb(1)} f_{\alpb(2)}}} \right) \xb^{\alpb(2)} \\
		&= \delta \left( \sqrt{f_{\alpb(1)}} \xb^{\frac{1}{2}\alpb(1)}- \sqrt{f_{\alpb(2)}} \xb^{\frac{1}{2}\alpb(2)}\right)^2+ f_{\alpb(1)} \left( 1- \delta \right) \xb^{\alpb(1)} + f_{\alpb(2)} \left( 1- \delta \right) \xb^{\alpb(2)},
	\end{align*}
	which proves the claim, as $\delta \in (0,1]$.
\end{proof}

For the proof of Proposition~\ref{prop:NotClosedLinearTransformations}, we also employ the following well-known inequality: Let $a,b \in \R$ and $\varepsilon >0$. Then, we have $0 \leq \left( \frac{a}{\sqrt{2\varepsilon}}-\sqrt{2\varepsilon}b \right)^2= \frac{a^2}{2\varepsilon}-2ab+2\varepsilon b^2 $. After rearranging, we derive \struc{\emph{Cauchy's Inequality with~$\varepsilon$}}, which states that $2ab \leq \frac{a^2}{2\varepsilon}+2\varepsilon b^2$. For $\varepsilon=\frac{1}{2}$, this inequality is simply referred to as \struc{\emph{Cauchy's Inequality}}.

\begin{proposition}\label{prop:NotClosedLinearTransformations}
	The cone $(\Sigma+C)_{n,2d}$ is closed under linear transformations of variables if and only if $(n,2d)$ is a Hilbert case.
\end{proposition}
\begin{proof}
	Since in the Hilbert cases, we have $\Sigma_{n,2d}=(\Sigma+C)_{n,2d} = P_{n,2d}$ and both $\Sigma_{n,2d}$ and $P_{n,2d}$ are closed under linear transformations, we only need to focus on the non-Hilbert cases.
	We apply a reduction strategy similar to Proposition \ref{prop:ReductionStrategySOS+SONC}. 
	Therefore, we denote by $A_k$ the $k$-th row of a matrix $A \in \R^{n \times n}$. Let $f \in P_{n,2d}$ and $A \in \R^{n \times n}$ such that $A_n=\bm{e}_n^T$. Thus, $f(A\cdot \xb)=f\left((A\cdot \xb)_1, \ldots (A\cdot \xb)_{n-1}, \varx_n\right)$, i.e. $A$ does not change the $n$-th variable. Further, define the block-diagonal matrix $\hat{A}\coloneqq\left( \begin{array}{cc}
		A & 0 \\
		0 & 1
	\end{array}\right) \in \R^{n+1 \times n+1}$.  
	If $f(A \cdot \xb) \not \in (\Sigma+C)_{n,2d}$, using Proposition \ref{prop:ReductionStrategySOS+SONC} we have 
	\begin{align*}
		&f(\hat{A} \cdot (\varx_1,\ldots, \varx_{n+1})^T)=f(A \cdot \xb, \varx_{n+1}) \not \in (\Sigma+C)_{n+1,2d},\quad \textup{and} \\
		&\left(\varx_n^2 f\right)(A \cdot \xb)=(\varx_n^2 f)((A\cdot \xb)_1, \ldots, (A\cdot \xb)_{n-1}, \varx_n)=\varx_n^2 \cdot f(A\cdot \xb) \not \in (\Sigma+C)_{n,2d+2}.
	\end{align*}
	Hence, it suffices to find $f$ and $A$ as above for the base cases $(n,2d)\in \{(3,6),(4,4)\}$.
	For $(n,2d)=(3,6)$, take the Motzkin form $M\in C_{3,6} \subseteq (\Sigma+C)_{3,6}$ and the invertible matrix
	\begin{align*}
		A\coloneqq\left(\begin{array}{ccc}
			1 & 0 & -1  \\
			0 & 1 & -1 \\
			0 & 0 & \phantom{-}1
		\end{array} \right).
	\end{align*}
	For a contradiction, assume that $\tilde{M}(\varx,\vary, \varz)\coloneqq M\left(A \cdot (\varx,\vary,\varz)^T\right)  \in (\Sigma+C)_{3,6}$, i.e.\@ $\tilde{M}=f_{\sos}+f_{\sonc}$ for some $f_{\sos}\in \Sigma_{3,6}$, $f_{\sonc} \in C_{3,6}$. Write $f_{\sos}= \sum_{i=1}^s g_i^2$, $f_{\sonc} = \sum_{j=1}^t h_j$	with $s,t \in \N$, $g_i \in H_{3,3}$, and nonnegative circuits $h_j \in H_{3,6}$ ($i \in [s], \ j \in [t]$).
	By the results of Section~\ref{subsec: SONCdecompWithoutTermCancellation}, we assume without loss of generality that $\supp(h_j) \subseteq \supp(f_{\sonc})$ for all $j \in [t]$. Further, since $M=\tilde{M}\left(A^{-1} (\varx,\vary,\varz)^T\right) \not \in \Sigma_{3,6}$ and $\Sigma_{3,6}$ is closed under linear transformations, we must have $\tilde{M} \not\in \Sigma_{3,6}$ and hence in particular $f_{\sonc} \not\in \Sigma_{3,6}$.
	Thus, we can assume without loss of generality that $h_j \not \in \Sigma_{3,6}$ ($j \in [t]$). 
	Moreover, Theorem \ref{thm:NewtonPolySumAndSOS} shows that $\New(g_i)\subseteq \frac{1}{2}\New(\tilde{M})$ and $ \New(h_j) \subseteq \New(\tilde{M})$ for all $i \in [s], j \in [t]$.
	Using $\tilde{M}(\varx, \vary, \varz)=M(\varx-\varz, \vary-\varz, \varz)$, we compute
	\begin{align}
		\supp(g_i) \subseteq& \frac{1}{2} \New(\tilde{M})\cap \N_0^n=\left\{\varx^2\vary, \ \varx^2\varz, \ \varx\vary^2, \ \varx\vary\varz, \ \varx\varz^2, \ \vary^2\varz, \  \vary\varz^2 \right\} &&(i \in [s]), \label{align:SOSpSONCnotClosedUnderLinTran01}\\
		V(h_j) \subseteq& \New(\tilde{M}) \cap 2\N_0^n=\left\{ \varx^4\vary^2, \ \varx^4\varz^2, \ \varx^2\vary^4, \ \varx^2\vary^2\varz^2, \ \varx^2\varz^4, \ \vary^4\varz^2, \ \vary^2\varz^4\right\} \quad &&(j \in [t]). \label{align:SOSpSONCnotClosedUnderLinTran02}
	\end{align}
	Now consider the monomials $\varx^4\vary^2, \ \varx^4\varz^2, \ -2\varx^4\vary\varz$ of $\tilde{M}$.  Assume that $\varx^4\vary\varz \in \supp(h_j)$, i.e. $\varx^4\vary\varz \in \cI(h_j)$ for some $j \in [t]$. 
	By \eqref{align:SOSpSONCnotClosedUnderLinTran02}, we must have $V(h_j)=\{\varx^4\vary^2, \ \varx^4\varz^2\}$. One can easily see that the assumptions of Lemma~\ref{lemma:PSDCircuitThreeNomial} are satisfied for $h_j$. This implies that $h_j \in \Sigma_{3,6}$, contradicting our choice of $h_j$'s. 
	For this reason, we have $\varx^4\vary\varz \not \in \supp(h_j)$ for all $j \in [t]$, which yields $(f_{\sos})_{\varx^4 \vary \varz}=\tilde{M}_{\varx^4\vary\varz}=-2$. Hence, using \eqref{align:SOSpSONCnotClosedUnderLinTran01} and Cauchy's Inequality, we obtain
	\begin{align*}
		-2=(f_{\sos})_{\varx^4 \vary \varz}= \sum_{i \in [s]}2 (g_i)_{\varx^2\vary} (g_i)_{\varx^2\varz} \geq& \sum_{i \in [s]}-2 \abs{(g_i)_{\varx^2\vary}} \abs{(g_i)_{\varx^2\varz}\vphantom{(g_i)_{\varx^2\vary}}} \\
		\geq& - \sum_{i \in [s]} (g_i)_{\varx^2\vary}^2 - \sum_{i \in [s]} (g_i)_{\varx^2\varz}^2  
		=-(f_{\sos})_{\varx^4 \vary^2}- (f_{\sos})_{\varx^4 \varz^2} \\
		\geq& -\tilde{M}_{\varx^4 \vary^2}- \tilde{M}_{\varx^4 \varz^2}=-1-1=-2,
	\end{align*}
	where the last inequality follows by using Proposition~\ref{prop:NewtonPolyPSD} and Theorem~\ref{thm:NewtonPolySumAndSOS}. Therefore, in the above, equality must hold everywhere, which leads to $(f_{\sos})_{\varx^4 \vary^2}=\tilde{M}_{\varx^4 \vary^2}$, $(f_{\sos})_{\varx^4 \varz^2}= \tilde{M}_{\varx^4 \varz^2}$ and hence $\varx^4\vary^2, \ \varx^4\varz^2 \not \in \supp(f_{\sonc})$. 
	Analogously, considering the monomials  $\varx^2\vary^4, \ \vary^4\varz^2, \ -2\varx\vary^4\varz$, we obtain  $\varx^2\vary^4, \ \vary^4\varz^2 \not \in \supp(f_{\sonc})$.
	A similar argument can be made by considering the monomials $\varx^4\varz^2, \ 4 \varx^2\varz^4, \ -4\varx^3\varz^3$, the only difference being that we use Cauchy's Inequality with $\varepsilon=\frac{1}{4}$, resulting in:
	\begin{align*}
		-4=(f_{\sos})_{\varx^3 \varz^3}= \sum_{i \in [s]}2 (g_i)_{\varx^2\varz} (g_i)_{\varx\varz^2} \geq& \sum_{i \in [s]}-2 \abs{(g_i)_{\varx^2\varz}} \abs{(g_i)_{\varx\varz^2}} \geq \sum_{i \in [s]} \left[ \left(-\frac{1}{2\varepsilon} \right)(g_i)_{\varx^2\varz}^2 - 2\varepsilon (g_i)_{\varx\varz^2}^2 \right]  \\
		=&  -2\sum_{i \in [s]} (g_i)_{\varx^2\varz}^2 - \frac{1}{2}\sum_{i \in [s]} (g_i)_{\varx\varz^2}^2  
		=-2 (f_{\sos})_{\varx^4 \varz^2}- \frac{1}{2}(f_{\sos})_{\varx^2 \varz^4} \\
		\geq& -2 \tilde{M}_{\varx^4 \varz^2}-\frac{1}{2} \tilde{M}_{\varx^2 \varz^4}=-2-2=-4.
	\end{align*}
	Again, equality holds everywhere and we obtain $\varx^4\varz^2,\varx^2\varz^4 \not \in \supp(f_{\sonc})$. Similarly, considering the monomials $\vary^4\varz^2, \ 4\vary^2\varz^4, -4 \vary^3\varz^3$ yields $\vary^4\varz^2,\  \vary^2\varz^4 \not \in \supp(f_{\sonc})$. Hence, using $\supp(h_j)\subseteq \supp(f_{\sonc})$ for all $j \in [t]$ and inclusion \eqref{align:SOSpSONCnotClosedUnderLinTran02}, we have $V(h_j) \subseteq \{\varx^2 \vary^2 \varz^2\}$, which is a contradiction since in this case $h_j$ is a monomial square and thus $h_j \in \Sigma_{3,6}$. This proves $\tilde{M}\not \in (\Sigma+C)_{3,6}$.	
	Analogously, for $(n,2d)=(4,4)$, take the Choi-Lam form $Q_1 \in C_{4,4} \subseteq (\Sigma+C)_{4,4}$ and the invertible matrix
	\begin{align*}
		B\coloneqq\left(\begin{array}{cccc}
			1 & 0 & 0 & -1  \\
			0 & 1 & 0 & -1 \\
			0 & 0 & 1 & -1 \\
			0 & 0 & 0 & \phantom{-}1
		\end{array} \right).
	\end{align*}
	For a contradiction, assume that there are $g_i \in H_{4,2}$, nonnegative circuits $h_j \in H_{4,4}$ with $i \in [s]$, $j \in [t]$, and $s,t \in \N$ such that $\tilde{Q_1}(\varx,\vary,\varz, \varw)\coloneqq Q_1\left(B \cdot (\varx,\vary,\varz, \varw)^T \right) =\sum_{i \in [s]} g_i^2+\sum_{j \in [t]} h_j$. As in the first part of the proof, we assume without loss of generality $\supp(h_j) \subseteq \supp\left(\sum_{j \in [t]} h_j\right)$ and $h_j \not \in \Sigma_{4,4}$ for all $j \in [t]$. Further, using Theorem~\ref{thm:NewtonPolySumAndSOS} and inspecting the terms of $\tilde{Q_1}(\varx,\vary,\varz, \varw)=Q_1(\varx-\varw, \vary-\varw, \varz-\varw, \varw)$ yields
	\begin{align*}
		\supp(g_i) \subseteq \frac{1}{2}\New(\tilde{Q_1}) \cap \N_0^n=&\left\{\varx\vary, \ \varx\varz, \ \varx\varw, \ \vary\varz, \ \vary\varw, \ \varz\varw, \ \varw^2\right\} && (i \in [s]), \\
		V(h_j) \subseteq \New(\tilde{Q_1}) \cap 2\N_0^n =& \left\{ \varx^2\vary^2, \ \varx^2\varz^2, \ \varx^2\varw^2, \ \vary^2\varz^2, \ \vary^2\varw^2, \ \varz^2\varw^2, \ \varw^4 \right\} && (j \in [t]).
	\end{align*}
	In a similar way as above, using Lemma~\ref{lemma:PSDCircuitThreeNomial} and applying Cauchy's Inequality with $\varepsilon=\frac{1}{4}$ to the monomials
	\begin{align*}
		\{2\varx^2\varw^2, \ 8\varw^4, \ -8\varx\varw^3\}, \
		\{2\vary^2\varw^2, \ 8\varw^4, \ -8\vary\varw^3\}, \ 
		\{2\varz^2\varw^2, \ 8\varw^4, \ -8\varz\varw^3\}
	\end{align*}
	shows $\varx^2\varw^2, \ \vary^2\varw^2, \ \varz^2\varw^2, \ \varw^4 \not \in \supp(h_j)$ and hence $V(h_j) \subseteq \{ \varx^2\vary^2, \ \varx^2\varz^2, \ \vary^2\varz^2\}$ for all $j \in [t]$. Therefore, we have $h_j \in P_{3,4}$ for all $j \in [t]$, since $h_j$ is a PSD circuit and the variable $\varw$ does not occur in $h_j$. However, by Hilbert~1888, we obtain $h_j \in \Sigma_{3,4}\subseteq \Sigma_{4,4}$, which is a contradiction. This proves $\tilde{Q_1}\not \in (\Sigma+C)_{4,4}$. 
\end{proof}

\section{A Hilbert 1888 analog for the SOS+SONC cone}\label{sec: MainResults}

The main result of this section is the following.

\begin{theorem}
	In the non-Hilbert cases, we have the following strict inclusions 
	\begin{align}\label{eq:mainResultSummary}
		(\Sigma_{n,2d}  \cup C_{n,2d} ) \subsetneq (\Sigma+C)_{n,2d}  \subsetneq P_{n,2d}.
	\end{align}
\end{theorem}

In Section \ref{subsec: SeparatingSOS+SONCfromBothSOSandSONC}, we establish the first strict inclusion of \eqref{eq:mainResultSummary}, see Proposition~\ref{prop:SOS+SONCnontrivialExtension}.
The second was already shown in \cite[Corollary~2.17]{Averkov_OptimalSizeofLMIS}, via an argument involving the \emph{semidefinite extension degrees} of the SOS and SONC cones. 
In Section \ref{subsec: SeparatingSOS+SONCfromPSD}, we prove this fact by providing explicit forms in $P_{n,2d}  \backslash (\Sigma+C)_{n,2d} $ for all non-Hilbert cases, see Proposition~\ref{prop:SOS+SONCHilbertAnalogon}.

\subsection{Separating the SOS+SONC cone from the PSD cone}\label{subsec: SeparatingSOS+SONCfromPSD} 

In this subsection, we prove that the SOS+SONC cone conincides with the PSD cone exactly in the Hilbert cases.

\begin{proposition}\label{prop:SOS+SONCHilbertAnalogon}
	The identity $(\Sigma+C)_{n,2d}  = P_{n,2d} $ holds if and only if $(n,2d)$ is a Hilbert case. 
\end{proposition}

For the only if part of Proposition \ref{prop:SOS+SONCHilbertAnalogon}, we need to find PSD forms, which are not SOS+SONC for all non-Hilbert cases. However, by the reduction strategy of Proposition \ref{prop:ReductionStrategySOS+SONC}, it suffices to consider the two base cases of ternary sextics and quarternary quartics.
Indeed, we prove in Lemma \ref{lemma:Robinson01} and Lemma \ref{lemma:Robinson02} that the Robinson forms $R_1, R_2$ and the Schm\"udgen form $S$ from Example \ref{ex:PSDnotSOS} are not SOS+SONC. 
To this end, we recall the arguments of Robinson and Schm\"udgen (based on Hilbert's original method) to show that $R_1, S$, and $R_2$ are not SOS. 
The main ideas are summarized in the subsequent proposition (see \cite{Reznick_OnHilbertsConstructionOfPositivePolynomials}).

\begin{proposition} \label{prop:CayleyBacharachRelation} The following Cayley-Bacharach Relations hold:
	\begin{enumerate}[(i),leftmargin=*]
		\item A form $q \in H_{3,3}$ vanishing on eight of the nine points in $\struc{\mathcal{X}} \coloneqq\{-1,0,1\} \times \{-1,0,1\} \times \{1\}$ must also vanish on the ninth.
		\item A form $q \in H_{3,3}$ vanishing on eight of the nine points in $\struc{\mathcal{X}'}\coloneqq\{-2,0,2\} \times \{-2,0,2\} \times \{1\}$ must also vanish on the ninth.
		\item A form $q \in H_{4,2}$ vanishing on seven of the eight points in $\struc{\mathcal{Y}}\coloneqq\{0,1\} \times \{0,1\} \times \{0,1\} \times \{1\}$ must also vanish on the eighth.
	\end{enumerate}
\end{proposition}

From Proposition \ref{prop:CayleyBacharachRelation} we deduce the following simple observations.

\begin{corollary}\label{cor:RobinsonNotSOSGeneralized}    
	\begin{enumerate}[(i)]
		\item Let $R_1=f_1+f_2$ be such that $f_1, f_2 \in P_{3,6}$. If $\mathcal{X}\subseteq \mathcal{Z}(f_2)$, then $f_1 \not \in \Sigma_{3,6}$.
		\item Let $S=f_1+f_2$ be such that $f_1, f_2 \in P_{3,6}$. If $\mathcal{X}'\subseteq \mathcal{Z}(f_2)$, then $f_1 \not \in \Sigma_{3,6}$.
		\item Let $R_2=f_1+f_2$ be such that $f_1, f_2 \in P_{4,4}$. If $\mathcal{Y}\subseteq \mathcal{Z}(f_2)$, then $f_1 \not \in \Sigma_{4,4}$.
	\end{enumerate}
\end{corollary}
\begin{proof}
	We only prove (i), as (ii) and (iii) can be shown analogously. A straightforward evaluation shows $\mathcal{X} \backslash \left\{(0,0,1)^T\right\} \subseteq \mathcal{Z}(R_1)$ and $(0,0,1)^T \not \in \mathcal{Z}(R_1)$, as $R_1(0,0,1)=1\neq 0$. 
	For a contradiction, assume $f_1 \in \Sigma_{3,6}$ with $f_1=\sum_{i=1}^s g_i^2$, $s \in \N$, $g_i \in H_{3,3}$ $(i \in [s])$. Since $R_1=\left(\sum_{i=1}^s g_i^2\right)+f_2$ and $g_i^2, f_2$ are PSD, we have in particular $\mathcal{Z}(R_1) \subseteq \mathcal{Z}(g_i)$ $(i \in [s])$. Thus, the inclusion $\mathcal{X}\backslash \left\{(0,0,1)^T\right\} \subseteq \mathcal{Z}(g_i)$ holds. 
	By Proposition \ref{prop:CayleyBacharachRelation}(i), $\mathcal{X} \subseteq \mathcal{Z}(g_i)$  for all $i \in [s]$. Therefore, we have $\mathcal{X} \subseteq \mathcal{Z}(f_1) \cap \mathcal{Z}(f_2)$, contradicting $R_1(0,0,1) \neq 0$.
\end{proof}

For $f_2=0$, the above Corollary shows in particular that $R_1, S$ and $R_2$ are not SOS.

\begin{lemma}\label{lemma:Robinson01}
	We have $R_1, S \not\in (\Sigma+C)_{3,6}$.
\end{lemma}
\begin{proof}
	Assume that $R_1 \in (\Sigma+C)_{3,6}$, i.e. $R_1=f_{\sos}+ f_{\sonc}$ for some $f_{\sos} \in \Sigma_{3,6}$, $f_{\sonc} \in C_{3,6}$. Let $f_{\sonc}=\sum_{j=1}^t h_j$ for some $t \in \N$ and nonnegative circuits $h_j \in H_{3,6}$ ($j \in [t]$). We claim that the inclusion $\mathcal{X} \subseteq \mathcal{Z}(f_{\sonc})$ holds.
	If $f_{\sonc}=0$, this is trivial. Otherwise, we can assume without loss of generality that $h_j \not \in \Sigma_{3,6}$ for all $j \in [t]$. In particular, we have $\mathcal{I}(h_j)=\{\betab(j)\}$ for some $\betab(j) \in \relint(\New(h_j))$. Further, Hilbert 1888 yields that for all $i \in [3]$, there is some $\alpb \in \supp(h_j)$ such that $\alpha_i \geq 1$. Since there is the strict convex combination $\betab(j)=\sum_{\alpb \in V(h_j)} \lambda_{\alpb \betab} \alpb$ with $\lambda_{\alpb \betab}>0$ and $\betab(j) \in \N_0^n$, it follows  $\beta(j)_i \geq 1$ for all $i \in [3]$.
	Hence, $h_j(0,\vary,\varz)$ is a sum of monomial squares and 
	\begin{align} \label{inequ:ProofRobinson01NotSOS+SONC}
		0\leq h_j(0,0,z) \leq  h_j(0,y,z)
	\end{align}
	holds for all $y,z \in \R, j \in [t]$.
	As observed above, $\mathcal{X}\backslash \left\{(0,0,1)^T\right\} \subseteq \mathcal{Z}(R_1)$. Since all $h_j$'s and $f_{\sos}$ are PSD, we have $\mathcal{Z}(R_1) \subseteq \mathcal{Z}(h_j)$ and thus $\mathcal{X}\backslash \left\{(0,0,1)^T\right\} \subseteq \mathcal{Z}(h_j)$  for all $j \in [t]$. For this reason, $(0,1,1)^T \in \mathcal{X}\backslash \left\{(0,0,1)^T\right\}$ and inequality \eqref{inequ:ProofRobinson01NotSOS+SONC} with $y=z=1$ shows that $h_j(0,0,1)=0$ ($j \in [t]$). We obtain $\mathcal{X}\subseteq \mathcal{Z}(h_j)$ for all $j \in [t]$ and hence $\mathcal{X} \subseteq \mathcal{Z}(f_{\sonc})$. 
	Corollary~\ref{cor:RobinsonNotSOSGeneralized}(i) yields that $f_{\sos} \not \in \Sigma_{3,6}$, which is a contradiction. This proves that $R_1 \not \in (\Sigma+C)_{3,6}$.
	
	Now assume that $S \in \Sigma_{3,6}$, i.e. $S=f_{\sos}+f_{\sonc}$ for some $f_{\sos} \in \Sigma_{3,6}, f_{\sonc} \in C_{3,6}$ and write $f_{\sonc}=\sum_{j=1}^t h_j$ for some $t \in \N$ and nonnegative circuits $h_j \in H_{3,6}$. We claim that $\mathcal{X}'\subseteq \mathcal{Z}(f_{\sonc})$. 
	Analogously as above, we assume without loss of generality that $h_j \not \in \Sigma_{3,6}$ and obtain $\beta(j)_i \geq 1$ for all $i \in [3]$, where $\betab(j) \in \relint(\New(h_j))$ is chosen such that $\mathcal{I}(h_j)=\{\betab(j)\}$. 
	Hence, $h_j(\varx,0,\varz)$ is a sum of monomial squares and in particular an even polynomial. 
	Further, note that $\mathcal{X}' \backslash \left\{(2,0,1)^T\right\} \subseteq \mathcal{Z}(S)$ holds and thus also $\mathcal{X}' \backslash \left\{(2,0,1)^T\right\} \subseteq \mathcal{Z}(h_j)$. Therefore, since $h_j(\varx,0,\varz)$ is even, we have $h_j(2,0,1)=h_j(-2,0,1)=0$, yielding $\mathcal{X}'\subseteq \mathcal{Z}(f_{\sonc})$. 
	Corollary~\ref{cor:RobinsonNotSOSGeneralized}(ii) shows that $f_{\sos} \not \in \Sigma_{3,6}$, which is a contradiction. This proves that $S \not \in (\Sigma+C)_{3,6}$.
\end{proof}

\begin{lemma}\label{lemma:Robinson02}
	We have $R_2 \not \in (\Sigma+ C)_{4,4}$.
\end{lemma}
\begin{proof}
	Assume, $R_2=f_{\sos}+ f_{\sonc}$ for some $f_{\sos} \in \Sigma_{4,4}$, $f_{\sonc} \in C_{4,4}$. Let $f_{\sonc}=\sum_{j=1}^t h_j$ for some $t \in \N$ and nonnegative circuits $h_j \in H_{4,4}$ ($j \in [t]$). For a contradiction, we show that $f_{\sonc}=0$. 
	Assume, $f_{\sonc} \neq 0$.
	Note that $\mathcal{Y}\backslash \left\{(1,1,1,1)^T\right\} \subseteq \mathcal{Z}(R_2)$. Hence, since all $h_j$'s and $f_{\sos}$ are PSD, we have in particular $\mathcal{Y}\backslash \{(1,1,1,1)^T\} \subseteq \mathcal{Z}(h_j)$ for all $j \in [t]$. 
	Further, since $f_{\sonc}\neq 0$, we can assume without loss of generality that $h_j \not \in \Sigma_{4,4}$ for all $j \in [t]$. As in the proof of Lemma \ref{lemma:Robinson01}, we obtain $\beta(j)_i\geq 1$ for all $i \in [4]$, where $\betab \in \relint(\New(h_j))$ is chosen such that $\mathcal{I}(h_j)=\{\betab(j)\}$. 
	Since $\abs{\betab(j)}=4$, we obtain $\betab\coloneqq \betab(j)=(1,1,1,1)^T$ for all $j \in [t]$.
	In addition, we have $(0,1,1,1)^T \in \mathcal{Z}(h_j)$ and hence
	\begin{align*}
		0=h_j(0,1,1,1)=\sum_{\alpb \in \supp(h_j), \alpha_1=0} (h_j)_{\alpb}=\sum_{\alpb \in V(h_j), \alpha_1=0} (h_j)_{\alpb},
	\end{align*}
	yielding $\{\alpb \in V(h_j): \alpha_1=0\}=\emptyset$. Since $V(h_j) \subseteq 2\N_0^n$, we get $\alpha_1 \geq 2$ for all $\alpb \in V(h_j)$. On the other hand, there is the convex combination $\betab=\sum_{\alpb \in V(h_j)} \lambda_{\alpb \betab} \alpb$ with $\lambda_{\alpb \betab}>0$ such that $\sum_{\alpb \in V(h_j)} \lambda_{\alpb \betab}=1$. We deduce $1=\beta_1=\sum_{\alpb \in V(h_j)} \lambda_{\alpb \betab} \alpha_1 \geq 2$, which is a contradiction. This shows $f_{\sonc}=0$ and hence $R_1=f_{\sos} \in \Sigma_{4,4}$, which is again a contradiction.
\end{proof}

With Lemma \ref{lemma:Robinson01} and \ref{lemma:Robinson02}, the proof of Proposition \ref{prop:SOS+SONCHilbertAnalogon} is completed. 

\begin{remark}
	One can show Lemma \ref{lemma:Robinson01} and \ref{lemma:Robinson02} without using the zero sets of $R_1, S$, and $R_2$, respectively.
	For this, one first proves $R_1, S \not \in \Sigma_{3,6}$ and $R_2 \not \in \Sigma_{4,4}$. This can be done without deploying the Cayley-Bacharach relations of Proposition \ref{prop:CayleyBacharachRelation}. Proofs can partially be found in \cite[Theorem~9.4]{Rajwade_Squares} and \cite{Reznick_OnHilbertsConstructionOfPositivePolynomials}. Second, one actually shows $R_1, S \not \in (\Sigma+C)_{3,6}$ and $R_2 \not \in (\Sigma+C)_{4,4}$ via a combinatorial investigation of the corresponding Newton polytopes. Detailed proofs will be available in \cite{Schick_PhDThesis}.
\end{remark}

\subsection{Separating the SOS+SONC cone from the union of the SOS and SONC cones} \label{subsec: SeparatingSOS+SONCfromBothSOSandSONC}

In this subsection, we prove that for all non-Hilbert cases, both the SOS and the SONC cones are proper subcones of the SOS+SONC cone, see Proposition \ref{prop:SOS+SONCnontrivialExtension}. We do so by providing in Lemma \ref{lemma:SeparateSOS+SONCFromSOSandSONC01} and \ref{lemma:SeparateSOS+SONCFromSOSandSONC02} forms in $(\Sigma+C)_{n,2d}  \backslash \left( \Sigma_{n,2d}  \cup C_{n,2d}  \right)$ for the two base cases of ternary sextics and quarternary quartics, respectively. 

\begin{lemma}\label{lemma:SeparateSOS+SONCFromSOSandSONC01}
	We have $f\coloneqq\frac{1}{2}  (\varz^3+2\varx\vary\varz+\varx^2\vary)^2+M(\varx, \vary, \varz) \ \in \ (\Sigma+C)_{3,6}  \backslash \left(\Sigma_{3,6}  \cup C_{3,6}  \right)$.
\end{lemma}
\begin{proof}
	Since $\frac{1}{2} (\varz^3+2\varx\vary\varz+\varx^2\vary)^2\in \Sigma_{3,6}$ and $M \in C_{3,6}$, we clearly have $f \in (\Sigma+C)_{3,6} $. Remains to show $f \not \in \left(\Sigma_{3,6}  \cup C_{3,6}  \right)$. First, we claim $f \not \in \Sigma_{3,6}$.  Assume $f= \sum_{i=1}^s g_i^2$ for some $s \in \N, \ g_i \in H_{3,3}$ $(i \in [s])$. 
	Factoring out shows
	\begin{align*}
		f&= \frac{1}{2}\left(\varz^6+4\varx\vary\varz^4+2\varx^2\vary\varz^3+4\varx^2\vary^2\varz^2+4\varx^3\vary^2\varz+\varx^4\vary^2 \right)+ \varx^4 \vary^2+\varx^2\vary^4+\varz^6-3\varx^2\vary^2\varz^2  \\
		&= \frac{3}{2} \varx^4\vary^2 +\varx^2\vary^4 + \frac{3}{2}\varz^6 - \varx^2 \vary^2 \varz^2 + 2\varx\vary\varz^4 +\varx^2\vary\varz^3 +2\varx^3\vary^2\varz.
	\end{align*}
	Theorem \ref{thm:NewtonPolySumAndSOS} yields the inclusion $\New(g_i) \subseteq \frac{1}{2} \New(f)= \conv \{\varx^2 \vary, \varx \vary^2, \varz^3\}$ $(i \in [s])$. Thus, we can write $g_i=g_{i,1}\varx^2 \vary+g_{i,2} \varx \vary^2+ g_{i,3} \varx \vary \varz+ g_{i,4} \varz^3$
	for some $g_{i,1}, \ldots, g_{i,4} \in \R$. Hence, we cannot obtain the monomial $-\varx^2 \vary^2 \varz^2$ with negative coefficient from the SOS decomposition $f= \sum_{i=1}^s g_i^2$, which is a contradiction.
	Second, we show that $f \not \in C_{3,6} $.
	From the explicit form of $f$ as above, we deduce
	\begin{align*}
		\Sc(f)=\left\{ \frac{3}{2} \varx^4\vary^2, \ \varx^2\vary^4, \ \frac{3}{2}\varz^6\right\} \quad \textup{and} \quad 
		\cI(f)=\left\{- \varx^2 \vary^2 \varz^2, \ 2\varx\vary\varz^4, \ \varx^2\vary\varz^3, \ 2\varx^3\vary^2\varz\right\}.
	\end{align*}
	Further, $\cR(f)=\emptyset$, which leads to $\sum_{{\betab} \in \cI(f)} \abs{f_{\betab}}=6 \gneq 4= \sum_{{\alpb} \in \Sc(f)\backslash \cR(f)} f_{\alpb}$.
	Thus, Theorem \ref{thm:NecessaryCondition} yields that $f \not \in C_{3,6}$. 
\end{proof}

\begin{lemma}\label{lemma:SeparateSOS+SONCFromSOSandSONC02}
	We have $f\coloneqq(\varx \vary+ \varx \varz + \vary \varz)^2+ \varw^4+Q_1(\varx, \vary, \varz, \varw) \in (\Sigma+C)_{4,4}  \backslash \left(\Sigma_{4,4}  \cup C_{4,4}  \right)$.
\end{lemma}
\begin{proof}
	Since $(\varx \vary+ \varx \varz + \vary \varz)^2+ \varw^4\in \Sigma_{4,4}$ and $Q_1 \in C_{4,4}$, we clearly have $f \in (\Sigma+C)_{4,4} $. Remains to show $f \not \in \left(\Sigma_{4,4}  \cup C_{4,4}  \right)$. First, we claim $f \not \in \Sigma_{4,4}$.  Assume $f= \sum_{i=1}^s g_i^2$ for some $s \in \N$ and $g_i \in H_{4,2}$ $(i \in [s])$. 
	Factoring out shows
	\begin{align*}
		f&=2\cdot (\varx^2 \vary^2+\varx^2 \varz^2+\vary^2 \varz^2+ \varw^4 -2\varw\varx\vary\varz+\varx\vary\varz^2+\varx^2\vary\varz+\varx\vary^2\varz).
	\end{align*}
	Theorem \ref{thm:NewtonPolySumAndSOS} yields the inclusion $\New(g_i) \subseteq \frac{1}{2} \New(f)= \conv\{\varx\vary, \varx\varz, \vary\varz, \varw^2\}$ $(i \in [s])$. Thus, we can write $g_i=g_{i,1}\varx \vary+g_{i,2} \varx \varz+ g_{i,3} \vary \varz+ g_{i,4} \varw^2$
	for some $g_{i,1}, \ldots, g_{i,4} \in \R$. Hence, we cannot obtain the monomial $-4 \varw \varx \vary \varz$ from the SOS decomposition $f=\sum_{i=1}^s g_i^2$, which is a contradiction.
	Second, we show that that $f \not \in C_{4,4} $.
	From the explicit form of $f$ as above, one can deduce
	\begin{align*}
		\Sc(f)=\left\{2 \varx^2 \vary^2, \ 2\varx^2 \varz^2, \ 2\vary^2 \varz^2, \ 2 \varw^4 \right\}, \quad 
		\cI(f)=\left\{-4\varw\varx\vary\varz, \ 2\varx\vary\varz^2, \ 2\varx^2\vary\varz, \ 2\varx^2\vary\varz\right\}.
	\end{align*}
	Further, $\cR(f)=\emptyset$, which leads to $\sum_{{\betab} \in \cI(f)} \abs{f_{\betab}} = 10 \gneq 8 = \sum_{{\alpb} \in \Sc(f)\backslash \cR(f)} f_{\alpb}$.
	Thus, Theorem~\ref{thm:NecessaryCondition} yields that $f \not \in C_{4,4}$. 
\end{proof}

Lemma \ref{lemma:SeparateSOS+SONCFromSOSandSONC01} and \ref{lemma:SeparateSOS+SONCFromSOSandSONC02} together with the reduction strategy for the SOS and SONC cones prove the following result.

\begin{proposition}\label{prop:SOS+SONCnontrivialExtension}
	The identity $\left(\Sigma_{n,2d}  \cup C_{n,2d} \right) = (\Sigma+C)_{n,2d}$ holds if and only if $(n,2d)$ is a Hilbert case.
\end{proposition}

\section{Conclusion}\label{sec:Conclusion}

The article introduces the cone $(\Sigma+C)_{n,2d}$ and studies its geometry. We establish necessary conditions for membership in $C_{n,2d}$ and show that $(\Sigma+C)_{n,2d}$ is a proper cone in $H_{n,2d}$, which is neither closed under multiplication nor under linear transformations of variables. 
Both properties are satisfied by $\Sigma_{n,2d}$ but not by $C_{n,2d}$.
Moreover, we separate $P_{n,2d}$ from $(\Sigma+C)_{n,2d}$ by showing that the Robinson forms are not contained in the latter cone. We also prove that adding suitable SOS forms to the Motzkin or Choi-Lam form leads to forms in $(\Sigma+C)_{n,2d} \backslash (\Sigma_{n,2d} \cup C_{n,2d})$. 

A natural next question that we address in an upcoming work concerns how to decide membership in $(\Sigma+C)_{n,2d}$ and exploit this cone in polynomial optimization.
In \cite{KaracaEtAl}, the authors combine the polynomial SOS cone and the signomial SAGE cone to tackle polynomial optimization problems. Indeed, their methods can be used to find SOS+SONC decompositions $f=f_1+f_2, \ f_1 \in \Sigma_{n,2d}, f_2 \in C_{n,2d}$ for \emph{even} polynomials $f$ with the additional assumption that $\supp(f_2) \subseteq \supp(f)$. We extend this approach to arbitrary polynomials without further assumptions on the support sets.

\section*{Acknowledgments}
We would like to thank the Mathematisches Forschungsinstitut Oberwolfach for its hospitality. Mareike Dressler was supported by a Simon's Visiting Professorship~2023. Her research stay was partially supported by the Simons Foundation and by the Mathematisches Forschungsinstitut Oberwolfach. 
Since 2024, Mareike is supported by the Australian Research Council Discovery Early Career Award DE240100674.
Salma Kuhlmann received support from the Ausschuss f\"ur Forschungsfragen, University of Konstanz.
Moritz Schick acknowledges support by the scholarship program of the University of Konstanz under the Landesgraduiertenf\"ordergesetz, the Studienstiftung des deutschen Volkes, as well as the Oberwolfach Leibniz Graduate Students Project.
We would also like to thank Janin Heuer and Khazhgali Kozhasov for various fruitful comments on different occasions. We thank the anonymous referee for valuable comments.

\bibliographystyle{alpha}
\bibliography{referencesFull}

\end{document}